\documentclass[a4paper,10pt,oneside,fleqn]{amsart}
\usepackage{amsthm}
\usepackage{amsmath}
\usepackage{tikz}
\usepackage{bm}
\usepackage{bbm}
\usepackage{amsfonts}
\usepackage{amssymb}
\usepackage{mathtools}
\usepackage{appendix}
\usepackage{mathrsfs}
\usepackage{setspace}
\usepackage{todonotes}
\usepackage{enumitem}
\usepackage[foot]{amsaddr}
\usepackage[pdfdisplaydoctitle,colorlinks,breaklinks,urlcolor=blue,linkcolor=blue,citecolor=blue]{hyperref} 
\usepackage[nameinlink,capitalise]{cleveref}

\newcommand{\E}{\mathbb{E}}

\newcommand{\N}{\mathbb{N}}

\renewcommand{\P}{\mathbb{P}}

\newcommand{\R}{\mathbb{R}}

\newcommand{\T}{\mathbb{T}}

\newcommand{\cA}{\mathcal{A}}

\newcommand{\cF}{\mathcal{F}}
\newcommand{\cG}{\mathcal{G}}

\newcommand{\cI}{\mathcal{I}}

\newcommand{\cK}{\mathcal{K}}
\newcommand{\cL}{\mathcal{L}}

\newcommand{\cQ}{\mathcal{Q}}
\newcommand{\cR}{\mathcal{R}}
\newcommand{\cS}{\mathcal{S}}
\newcommand{\cT}{\mathcal{T}}

\newcommand{\mf}{M}
\newcommand{\am}{\hat{\bm{\alpha}}^-_d}
\newcommand{\ap}{\hat{\bm{\alpha}}^+_d}
\newcommand{\apm}{\hat{\bm{\alpha}}^\pm_d}
\newcommand{\sm}{\hat{\bm{s}}^-_d}
\renewcommand{\sp}{\hat{\bm{s}}^+_d}

\newcommand{\ams}{\hat \alpha_{d,s}^-}
\newcommand{\aps}{\hat \alpha_{d,s}^+}
\newcommand{\apms}{\hat \alpha_{d,s}^\pm}
\newcommand{\sma}{\hat s_{d,\alpha}^-}
\newcommand{\spa}{\hat s_{d,\alpha}^+}
\newcommand{\spma}{\hat s_{d,\alpha}^\pm}

\DeclareMathOperator{\trace}{Tr}

\DeclareMathOperator{\re}{Re}

\let\div\relax
\DeclareMathOperator{\div}{div}

\DeclareMathOperator{\HS}{HS}
\DeclareMathOperator*{\argmin}{\arg\!\min}

\newcommand{\dd}{\,\mathrm{d}}
\newcommand{\eps}{\varepsilon}
\renewcommand{\setminus}{\smallsetminus}
\newcommand{\loc}{\mathrm{loc}}

\newcommand{\set}[1]{\left\{#1\right\}}

\newcommand{\gamhalf}[1]{\Gamma\left(\frac{#1}{2}\right)}

\DeclareMathOperator{\res}{Res}

\newtheorem{theorem}{Theorem}[section]
\newtheorem{definition}[theorem]{Definition}
\newtheorem{hypothesis}[theorem]{Hypothesis}
\newtheorem{corollary}[theorem]{Corollary}
\newtheorem{lemma}[theorem]{Lemma}
\newtheorem{proposition}[theorem]{Proposition}

\theoremstyle{remark}
\newtheorem{remark}[theorem]{Remark}

\numberwithin{equation}{section}

\renewcommand{\le}{\leqslant}

\newcommand{\bracn}{\langle n \rangle}

\newenvironment{acknowledgements}{%
  \begin{abstract}
}{%
  \end{abstract}
}

\title[Anomalous Regularization in Kazantsev-Kraichnan Model]{Anomalous Regularization in Kazantsev-Kraichnan Model}
\author[M. Bagnara]{Marco Bagnara}
\address{Scuola Normale Superiore, Piazza dei Cavalieri 7, 56126 Pisa, Italy}
\email{marco.bagnara@sns.it}
\author[F. Grotto]{Francesco Grotto}
\address{Dipartimento di Matematica, Università di Pisa, Largo Bruno Pontecorvo 5, 56127 Pisa, Italy}
\email{francesco.grotto@unipi.it}
\author[M. Maurelli]{Mario Maurelli}
\address{Dipartimento di Matematica, Università di Pisa, Largo Bruno Pontecorvo 5, 56127 Pisa, Italy}
\email{mario.maurelli@unipi.it}

\begin{document}

\begin{abstract}
This work investigates a passive vector field which is transported and stretched by a divergence-free Gaussian velocity field, delta-correlated in time and poorly correlated in space (spatially nonsmooth). Although the advection of a scalar field (Kraichnan's passive scalar model) is known to enjoy regularizing properties, the potentially competing stretching term in vector advection may induce singularity formation. We establish that the regularization effect is actually retained in certain regimes. While this is true in any dimension $d\ge 3$, it notably implies a regularization result for linearized 3D Euler equations with stochastic modeling of turbulent velocities, and for the induction equation in magnetohydrodynamic turbulence.
\end{abstract}

\maketitle


\tableofcontents

\section{Introduction}\label{sec:introduction}

Consider the 3D vector advection equations 
\begin{equation}\label{eq:VAE}
    \begin{cases}
        \partial_t \mf= \nabla\times(u \times\mf),\\
        \nabla\cdot \mf=0,
    \end{cases}
\end{equation}
where $u$ is a divergence-free velocity field. This model arises in at least two contexts: 

\begin{description}[labelindent=0cm, leftmargin=0.2cm]
    \item[Magnetic Induction] $\mf$ is the magnetic field generated by a charged fluid with velocity field $u$, assuming that Lorentz force is negligible (infinite electric conductivity);
    \item[Incompressible Euler Equations] $\mf$ is the vorticity field of a constant density fluid, the velocity field $u$ and $\mf$ satisfying $\mf = \nabla \times u$.
\end{description}

Regularity and singularity formation in this PDE model are important mathematical problems.
In the case of 3D Euler equations with smooth initial datum the PDE is locally well-posed,
but singularities can develop at finite time, see the celebrated 
work by Elgindi \cite{elgindi2021finite} on blow-up for $M\in C^\alpha$.
Writing the first equation in \eqref{eq:VAE} as
\begin{equation*}
    \partial_t \mf+ (u\cdot \nabla)\mf = (\mf\cdot \nabla) u,
\end{equation*}
the common understanding is that the stretching term $(\mf\cdot \nabla) u$ favors the development of singularities, while the advection term $(u\cdot \nabla)\mf$ can stabilize the flow \cite{ElgJeo2020b}.

Following the statistical approach to fluid mechanics, we can consider \emph{random} ensembles of velocity fields $u$ with \emph{prescribed} statistics and study the regularity properties of the (random) solution $\mf$. In this context, the regularity (especially for negative Sobolev indices) is not only related to classical well-posedness questions in PDEs, but also to the covariance function $\psi(x,y)=\E[\mf(x)\mf(y)]$ and the energy spectrum $E(\rho)$ of the solution. Indeed, at least when all statistics are isotropic, we have
\begin{align*}
    \E\|\mf\|_{\dot{H}^{-s}}^2 &\approx \int_{\R^3} \rho^{-2s} E(\rho) \dd \rho \\
    &\approx \iint_{\R^3\times \R^3} |x-y|^{-3+2s} \psi(x,y) \dd x\dd y.
\end{align*}

In order to perform explicit computations in the context of dynamo theory, Kazantsev \cite{kazantsev1968} proposed to model the turbulent velocity field with a Gaussian random vector field $u^\cK$, delta-correlated in time and having a power law covariance spectrum. 
The corresponding stochastic PDE, that we consider here in any space dimension $d\ge 2$,
\begin{equation}\label{eq:KK}
    \partial_t\mf +(u^\cK\cdot\nabla) \mf = (\mf\cdot\nabla) u^\cK
\end{equation}
is also known as \emph{Kazantsev-Kraichan model} \cite{vincenzi2002} because of its close relation with Kraichnan's passive scalar advection model introduced in the same year \cite{Kraichnan1968},
\begin{align}\label{eq:K_passive}
    \partial_t \theta +u^\cK\cdot\nabla \theta =0,
\end{align}
where a scalar field $\theta$ (e.g. the temperature of the fluid) is advected by a Gaussian random vector field $u^\cK$ of the same type of above. Both models in fact elaborate on works of Batchelor, see respectively \cite{batchelor1950,batchelor1959} for the advection of (magnetic) vector fields and scalars.

A particular case of the Kazantsev-Kraichnan model is when the Gaussian velocity field $u^\cK$ has covariance function given by
\begin{align}
\begin{aligned}\label{eq:covariance_intro}
    &\E[u^\cK(t,x)u^\cK(s,y)^\top] = \delta_0(t-s) Q(x-y),\\
    &\widehat{Q}(n) = \langle n \rangle^{-d-2\alpha} \left(I-\frac{n n^\top}{|n|^2}\right), \quad \alpha\in (0,1),
\end{aligned}
\end{align}
where $\langle n \rangle = (1+|n|^2)^{1/2}$. This covariance function corresponds to an almost $C^{\alpha}$ space regularity of $u^\cK$ and mimics the irregularity of a turbulent velocity field. In this \emph{nonsmooth} context, the irregularity of $u^\cK$ is a challenge to well-posedness and regularity theory for the Kazantsev-Kraichnan model \eqref{eq:KK}. This issues arises both from the transport and the stretching term, as it is well-known that an irregular deterministic vector field can produce non-uniqueness of solutions to the associated passive scalar equation (see e.g. \cite{DiPernaLions1989}).

For the Kraichnan model of passive scalars \eqref{eq:K_passive}, that is without stretching, well-posedness has been proved by E and Vanden-Eijnden \cite{EVan2001} and Le Jan and Raimond \cite{LeJRai2002}, but with a picture which is completely different from standard well-posedness of deterministic transport equations: at the Lagrangian level, as already understood in physical literature (e.g. Bernard, Gawedzki and Kupiainen \cite{BeGaKu1998}), particles undergo splitting, which in turns brings anomalous dissipation of the passive scalars, see Rowan \cite{Rowan2023} for a rigorous proof and a quantitative statement. Moreover, splitting brings an anomalous regularization of the passive scalars, as Coghi and Maurelli \cite{coghi2023existence} and Galeati, Grotto and Maurelli \cite{galeati2024anomalous} have shown (see also Eyink and Xin \cite{EyinkXin2000} on two-point motion regularization).

Compared to the passive scalar case, the addition of the stretching term in \eqref{eq:KK} can have a disruptive effect on the dynamics: the solution field $\mf$ is multiplied by $\nabla u^\cK$ which is only a distribution in space, and the $L^2$ norm of $\mf$ is no longer controlled; in fact we expect the $L^2$ norm to blow up (see \cref{sec:L2_blow_up}). At a (formal) Lagrangian level, $\mf$ is stretched by the derivative of the flow, namely
\begin{equation*}
    M(t,x) = D\Phi_t(\Phi_t^{-1}(x)) M(0,\Phi_t^{-1}(x)),
\end{equation*}
however the flow itself $\Phi$ (of particles driven by $u^\cK$) is not defined due to splitting. Hence the regularization effect coming from the transport term could be destroyed by the stretching effect, making the Kazantsev-Kraichnan model \eqref{eq:KK} ill-posed.

Our main result shows that, in a certain regime, this does not happen, and the Kazantsev-Kraichnan model \eqref{eq:KK} keeps an anomalous regularization property, for negative Sobolev-valued solutions:
\begin{theorem}[Informal statement, see \cref{thm:main} for the precise result]\label{thm:main_intro}
    Let $d\ge 3$, 
    $\alpha\in (0,\ap\wedge 1)$ and $s\in (\sma,\spa \wedge (d/2))$, where $\ap$, $\sma$ and $\spa$ are defined in \cref{hp:main_hp}. Let $u^\cK$ be a centered Gaussian velocity field with covariance function \eqref{eq:covariance_intro}. For every $\mf_0$ in $\dot{H}^{-s}_{\div =0}$, there exists a unique solution $\mf$ to \eqref{eq:KK} in the class $L^2_{t,\omega}(\dot{H}^{-s+1-\alpha}_{\div =0})$. This solution satisfies
    \begin{equation*}
        \sup_{t\in [0,T]}\E[\|\mf_t\|_{\dot{H}^{-s}}^2] +\eta_{d,s,\alpha}\int_0^T \E[\|\mf_r\|_{\dot{H}^{-s+1-\alpha}}^2] \dd r \le e^{\rho_{d,s,\alpha}T} \|\mf_0\|_{\dot{H}^{-s}}^2.
    \end{equation*}
\end{theorem}
\begin{remark}
    For $d=3$, the conditions on $\alpha$ and $s$ become
\begin{equation*}
    \alpha\in \left(0,\frac12\right), \quad s\in \left(\sma,\frac32\right),\quad \sma=\frac 74 - \frac \alpha2 -\frac {\sqrt3}{4} \sqrt{3-4\alpha-4\alpha^2}.
\end{equation*}
\end{remark}

Note that the solutions to \eqref{eq:KK} from \cref{thm:main_intro} take values only in the space of distributions. One may ask what happens to the $L^2$ norm of the solution: the heuristic computations in \cref{sec:L2_blow_up} suggest that, even for a smooth initial condition, the $L^2$ norm blows up instantaneously.

Let us briefly explain the heuristics and the strategy of the proof of \cref{thm:main_intro}. For the heuristics, as in \cite{coghi2023existence}, we replace $Q$ by its first order approximation (so-called self-similar case):
\begin{equation}\label{eq:Q_approx}
    Q(0)-Q(x) \approx \beta_L |x|^{2\alpha}\frac{xx^\top}{|x|^2} +\beta_N |x|^{2\alpha}\left(I-\frac{xx^\top}{|x|^2}\right),\quad \beta_N = \left(1+\frac{2\alpha}{d-1}\right)\beta_L.
\end{equation}
This approximation, while not rigorous in computations, already allows to get a good understanding of the evolution of the model. We call $G_s$ the Green kernel associated with $(-\Delta)^{-s}$. We write the equation for $\E\|\mf\|_{\dot{H}^{-s}}^2=\E\langle G_s\ast \mf,\mf\rangle$, getting
\begin{align}\label{eq:heuristics}
    \frac{\dd}{\dd t}\E\|\mf_t\|^2_{\dot H^{-s}} 
    &= \E\langle \trace((Q(0)-Q) D^2G_s) \ast \mf, \mf\rangle \\ \notag
    &\quad +\E\langle \big((\trace (Q(0)-Q)) D^2 G_s  \big)\ast \mf ,\mf \rangle \\ \notag
    &\quad -2\E\langle ((Q(0)-Q) D^2G_s)\ast \mf, \mf \rangle \\ \notag
    &=: T_{tra}+T_{str}+T_{mix}    
\end{align}
The first term $T_{tra}$ in the right-hand side is due to the transport term, while the second and third terms $T_{str}$ and $T_{mix}$ come also from the stretching term. As known from \cite{coghi2023existence} and \cite{galeati2024anomalous}, for the transport term $T_{tra}$ we get a negative, regularizing contribution: for some $c>0$,
\begin{equation*}
    \E\langle \trace((Q(0)-Q) D^2G_s) \ast \mf, \mf\rangle = -c \E\langle G_{s-1+\alpha} \ast \mf, \mf\rangle = -c\E\|\mf_t\|^2_{\dot H^{-s+1-\alpha}}.
\end{equation*}
Compared to the computations for $T_{tra}$, a relevant difference is that the kernels $(\trace (Q(0)-Q)) D^2 G_s$ and $(Q(0)-Q) D^2G_s$ in the terms $T_{str}$ and $T_{mix}$ are matrix-valued, hence they do not need to have a sign. The key observation to overcome this issue is the following: by isotropy of $Q$ and $G_s$, each kernel can be written as
\begin{equation*}
    f(x)I+ a D^2G_{s+\alpha}
\end{equation*}
for suitable scalar function $f$ and real constant $a$; the kernel $D^2G_{s+\alpha}$, when tested with divergence-free functions, gives no contribution, hence we are reduced to the scalar case. After lengthy but elementary computations, we find
\begin{equation*}
    \frac{\dd}{\dd t}\E\|\mf_t\|^2_{\dot H^{-s}}  = -\tilde{\eta}\,\E\|\mf_t\|^2_{\dot H^{-s+1-\alpha}}
\end{equation*}
and the constant $\tilde{\eta}$ is positive if and only if the assumptions on $d$, $s$ and $\alpha$ in \cref{hp:main_hp} are met.

Making these heuristic computations rigorous requires a control, in suitable Sobolev norms, of the higher order terms in the approximation \eqref{eq:Q_approx}, something which is quite technical (it has been done in \cite{coghi2023existence} for the transport term and $s=1$, $d=2$). Instead, the rigorous proof takes the route in \cite{galeati2024anomalous}, controlling directly the right-hand side of \eqref{eq:heuristics} in Fourier modes by complex analysis tools. In this context, due again to isotropy, the Fourier transform of the kernels takes the form
\begin{equation*}
    g(n)I +h(n)\frac{nn^\top}{|n|^2}
\end{equation*}
for scalar functions $g$ and $h$, and the second terms gives no contribution, because $n\cdot \widehat{\mf}(n)=0$ due to the divergence-free condition on $\mf$.

The \emph{nonsmooth} Kazantsev-Kraichnan model has been studied in the physical literature, in relation to the so-called dynamo effect (namely when turbulent transport results in a substantial growth of the passive vector field), see the works by Kazantsev himself \cite{kazantsev1968}, Vergassola \cite{Vergassola1996} and Vincenzi \cite{vincenzi2002}, where the absence of dynamo effect is shown, at a physical level, for $\alpha\in (0,1/2)$ at least in the self-similar case.

In the mathematical literature, to our knowledge, our work is the first one dealing with the \emph{nonsmooth} Kazantsev-Kraichnan model. However, the case of a smooth noise has been previously considered. In \cite{BaxendaleRozovski1993}, Baxendale and Rozovskii characterized dynamo effect and intermittency of the smooth Kazantsev-Kraichnan model in terms of the properties of the associated Lagrangian flow. Holm \cite{holm2015variational} introduced a variational framework (the so-called stochastic advection by Lie transport) to derive (also) vector advection equation driven by random Gaussian delta-correlated vector fields (not necessarily with power law spectrum); this setting has been used also for nonlinear equations perturbed by stochastic vector advection, see e.g. Drivas, Holm and Leahy \cite{DrivasHolmLeahy2020} and Crisan, Flandoli and Holm \cite{CrisanFlandoliHolm2019}. More recently, Butori and Luongo \cite{ButoriLuongo2024} and Butori, Flandoli and Luongo \cite{ButoriFlandoliLuongo2024} have considered the Kazantsev-Kraichnan model with a supplementary viscosity in an It\^o-Stratonovich diffusion limit in $\T^3$ and a thin domain respectively. 
Such limit procedure, inspired by the the work \cite{Galeati2020}, involves a sequence of smooth noises formally approaching a singular one with covariance concentrated at the origin. 
They showed convergence to a determinist equation with an additional effective viscosity and alpha-term (related with the dynamo effect) based on the the dimension of the domain and the shape of the covariance. 
Finally, Coti Zelati and Navarro-Fern{\'a}ndez \cite{CotiZelatiNavarroFernandez2024} have considered a passive vector fields advected by a randomized ABC flow and have proved almost sure ideal kinematic dynamo by showing a chaotic behaviour of the Lagrangian flow.

The paper is organized as follows. In \cref{sec:Heuristic_Derivation}, we provide a heuristic derivation of the main result, analyze the range of parameters for which anomalous regularization occurs, and, in \cref{sec:L2_blow_up}, explain why we expect an instantaneous blow-up of the $L^2$ norm. In \cref{sec:Setting_and_Main_Result}, we rigorously introduce the Kazantsev-Kraichnan noise, study the transport plus stretching operator and conclude by defining the notion of solution and stating the main result of the paper, \cref{thm:main}. In \cref{sec:Viscous_approximations}, we show the existence and uniqueness of viscous approximations. \cref{sec:Main_Bound} is dedicated to proving the main bound of the work, which establish anomalous regularization in the \emph{nonsmooth} Kazantsev-Kraichnan model. In \cref{sec:Proof_Existence}, we prove the existence of solutions using a classical compactness and convergence argument, while in \cref{sec:ProofUniqueness} we establish pathwise uniqueness. \cref{sec:Appendix} contains a formal derivation of the It\^o-Stratonovich correction, along with auxiliary and technical results.


\subsection{Notation}
The Fourier transform of $f$ is denoted by
$\widehat{f}$ and we use the convention
    \begin{equation*}
    \widehat f(\xi)=(2\pi)^{-\frac{d}{2}}\int_{\R^d} f(x) e^{-i\xi\cdot x} \dd x,\quad
    f(x)=(2\pi)^{-\frac{d}{2}} \int_{\R^d} \widehat f(\xi) e^{i\xi\cdot x} \dd\xi.  
\end{equation*}
We recall (denoting by $\overline{\widehat{g}}$ the complex conjugate of $\widehat{g}$) that
\begin{equation*}
    \widehat{f\ast g} = (2\pi)^{\frac{d}{2}} \widehat f\,\widehat g, \quad \widehat{fg} = (2\pi)^{-\frac{d}{2}} \widehat{f}\ast\widehat{g}, \quad
    \langle f,g \rangle = \langle \widehat{f},\overline{\widehat{g}} \rangle.
\end{equation*}

For $s\in \R$, the inhomogeneous Sobolev space $H^s(\R^d)$ is defined by
\begin{equation*}
    H^s(\R^d) = \Big\{f\in \mathcal{S}'(\R^d): \hat f\in L^1_{loc}(\R^d),\,\| f\|_{H^s}^2 := \int_{\R^d} \langle \xi\rangle^{2s} |\hat f(\xi)|^2 \dd \xi <\infty\Big\},
\end{equation*}
where $\langle \xi\rangle \coloneqq (1+|\xi|^2)^\frac12$. Similarly, the homogeneous Sobolev space $\dot H^s(\R^d)$ is defined by replacing $\langle \xi \rangle$ by $|\xi|$ in the definition of $\| \cdot\|_{H^s}$; $\dot{H}^s(\R^d)$ is a Hilbert space for $s<d/2$.

The subscript $\div$ denotes the subspaces of divergence-free vector fields (in the distributional sense). For $k\in \R^d$, $k\neq 0$, the notation $P^\perp_k=I-kk^\top/|k|^2$ is used for the projection on the space orthogonal to $k$.

We often use the subscripts $t$ and $\omega$ to denote the functional spaces on $[0,T]$ and $\Omega$ respectively, while we use no subscript for the functional spaces on $\R^d$: for example $L^2_{t,\omega}(H^s_{\div})$, $L^\infty_t(L^2_\omega(H^s_{\div}))$ denote the spaces $L^2([0,T]\times \Omega;H^s_{\div}(\R^d))$, $L^\infty([0,T];L^2(\Omega;H^s_{\div}(\R^d)))$ respectively.

Given $H$, $H'$ Hilbert spaces, $L(H,H')$ denotes the space of linear bounded operators from $H$ to $H'$, while $\HS(H,H')$ denotes the space of Hilbert-Schmidt operators from $H$ to $H'$, endowed with the Hilbert-Schmidt norm $\|T\|_{\HS}:= \trace[T^*T]^{1/2}$.

\section{Heuristic Derivation in the Statistically Self-similar Case}\label{sec:Heuristic_Derivation}

In the case of the so-called statistically self-similar Kraichnan noise heuristic computations allow to guess the anomalous regularity gain and the range of parameters for which it takes place.

Let us consider a passive divergence-free vector field $\mf$ with values in $\R^d$ which is transported and stretched by a random divergence-free velocity field $\dot  W$. We assume $W$ to be a spatially homogeneous Wiener process satisfying for $x,y \in \R^d$ and $s,t\ge 0$,
\begin{equation*}
	\E[W(t,x)W(s,y)^\top] = (t\wedge s) \,Q(x-y),
\end{equation*}
where $Q$ will be specified later. Consider the stochastic advection equation
\begin{equation}\label{eq:SPDE_Stratonovich_sss} 
	\dd \mf_t +  \circ \mathrm{d} W_t \cdot \nabla \mf_t - \mf_t \cdot \nabla \circ \mathrm{d} W_t=0,
\end{equation}
where $\circ$ denotes the Stratonovich stochastic integration. This is formally the correct interpretation of \eqref{eq:KK} for Gaussian vector fields delta-correlated in time (see \cref{subsec:transport_stretching}). The SPDE \eqref{eq:SPDE_Stratonovich_sss} can be converted, at least formally, in It\^o form,
\begin{equation}\label{eq:SPDE_Ito_sss}
	\dd \mf_t + \mathrm{d} W_t \cdot \nabla \mf_t - \mf_t \cdot \nabla \mathrm{d} W_t= \frac{c_0}{2} \Delta \mf_t \dd t,
\end{equation}
where $c_0\operatorname{I}=Q(0)$. \cref{sec:Ito-Stratonovich_correction} reports a formal derivation of the It\^o-Stratonovich correction.
We are interested in the temporal evolution of negative Sobolev norms of the vector field $\mf$. For $s\in(0,d/2)$
\begin{equation*}
    \|\mf\|^2_{\dot H^{-s}}=\|(-\Delta)^{-s/2}\mf\|^2_{L^2}= \langle G_s \ast \mf, \mf \rangle,
\end{equation*}
where $G_s(x)=c_{d,s}|x|^{-d+2s}$ is the Riesz potential. If we decompose the noise $\mathrm{d} W_t = \sum_{k} \sigma_k \dd W^k_t$ and apply It\^o formula, we obtain
\begin{multline}\label{eq:Ito_formula_sss}
    \dd \|\mf_t\|^2_{\dot H^{-s}} 
	+ 2 \sum_{k} \langle G_s \ast \mf_t, \sigma_k\cdot \nabla \mf_t - \mf_t \cdot \nabla \sigma_k \rangle \dd W^k_t- c_0 \langle G_s \ast \mf_t,\Delta \mf_t \rangle \dd t\\
	= \sum_{k} \langle G_s\ast (\sigma_k\cdot\nabla \mf_t-\mf_t\cdot\nabla\sigma_k),(\sigma_k\cdot\nabla \mf_t-\mf_t\cdot\nabla\sigma_k)\rangle \dd t.
\end{multline}
Being $G_s$ an even function, we have
\begin{multline*}
    \sum_{k} \langle G_s\ast (\sigma_k\cdot\nabla \mf-\mf\cdot\nabla\sigma_k),(\sigma_k\cdot\nabla \mf-\mf\cdot\nabla\sigma_k)\rangle \\
		= \sum_{k\ge 1} \langle G_s\ast (\sigma_k\cdot\nabla \mf),\sigma_k\cdot\nabla \mf\rangle - 2 \sum_{k\ge 1} \langle G_s\ast (\sigma_k\cdot\nabla \mf),\mf\cdot\nabla\sigma_k\rangle\\ 
		+\sum_{k\ge 1} \langle G\ast (\mf\cdot\nabla\sigma_k),\mf\cdot\nabla\sigma_k\rangle.
\end{multline*}
Then, leveraging the divergence-free condition for $\mf$ and $\sigma_k$, integration by parts, the parity of $G_s$ and $Q$ and an analogous to \cref{lem:sigma_k}, we obtain (formally)
\begin{align}
	&\sum_{k\ge 1} \langle G_s\ast (\mf\cdot\nabla\sigma_k),\mf\cdot\nabla\sigma_k\rangle = - \langle \big((\trace Q) D^2 G_s  \big)\ast \mf ,\mf \rangle, \label{eq:stretching_term_sss}\\
	&\sum_{k\ge 1} \langle G_s\ast (\sigma_k\cdot\nabla \mf),\sigma_k\cdot\nabla \mf\rangle = - \langle \trace(Q D^2G_s) \ast \mf, \mf\rangle, \label{eq:transport_term_sss}\\
	&\sum_{k\ge 1} \langle G_s\ast (\sigma_k\cdot\nabla \mf),\mf\cdot\nabla\sigma_k\rangle 
	= - \langle (Q D^2G_s)\ast \mf, \mf \rangle.\label{eq:mixed_term_sss}
\end{align}
We refer to \cref{lem:smooth_Ito_term} for a more detailed derivation of the previous identities.
Integrating by parts and recalling that $Q(0)=c_0 I$,
\begin{equation}\label{eq:laplacian_term_sss}
	c_0 \langle G_s \ast \mf_t,\Delta \mf_t \rangle = \langle \trace(Q(0)D^2G_s)\ast \mf,\mf \rangle
\end{equation}
Hence, substituting \eqref{eq:stretching_term_sss}-\eqref{eq:laplacian_term_sss} in \eqref{eq:Ito_formula_sss}, we conclude that 
\begin{align*}
	\begin{split}
		\dd \|\mf_t\|^2_{\dot H^{-s}} 
		&= \big( \langle \trace((Q(0)-Q) D^2G_s) \ast \mf, \mf\rangle\\
		&\quad- \langle \big((\trace Q) D^2 G_s  \big)\ast \mf ,\mf \rangle +2 \langle (Q D^2G_s)\ast \mf, \mf \rangle \big) \dd t +\dd N_t,
	\end{split}
\end{align*}
where $N_t$ is a local martingale. However, being $Q(0)$ a multiple of the identity and  $\langle D^2G\ast \mf, \mf\rangle=0$ by integration by parts and divergence-free condition, we can rewrite the previous identity as
\begin{align}\label{eq:time_evolution_sss}
	\begin{split}
		\dd \|\mf_t\|^2_{\dot H^{-s}} 
		&= \big(\langle \trace((Q(0)-Q) D^2G_s) \ast \mf, \mf\rangle\\
		&\qquad+\langle \big((\trace (Q(0)-Q)) D^2 G_s  \big)\ast \mf ,\mf \rangle\\
		&\qquad -2 \langle ((Q(0)-Q) D^2G_s)\ast \mf, \mf \rangle \big) \dd t +\dd M_t.\\
		&\eqqcolon - F_t \dd t +\dd M_t.
	\end{split}
\end{align}
Identity \eqref{eq:time_evolution_sss} highlights how the relevant quantity in the time evolution $\dot H^{-s}$ is the function $Q(0)-Q(\cdot)$. It is here that the statistically self-similar noise turns out to be very convenient to explicitly compute the right-hand side of \eqref{eq:time_evolution_sss}.

The statistically self-similar Kraichnan noise is formally obtained by replacing \eqref{eq:Q_spectrum} with
\begin{equation}\label{eq:Qsss_spectrum}
	\widehat{Q}(n) = |n|^{-d-2\alpha} \left(I-\frac{n n^\top}{|n|^2}\right).
\end{equation}
Due to the lack of integrability at the origin $Q(0) = +\infty$, which seems to invalidate many of the previous computations, starting for instance from the unboundedness of the It\^o-Stratonovich correction. However, as one can formally derive following e.g. \cite[Section 10]{LeJRai2002} and \cite[Section 6]{GGM2024}, we have
\begin{align} \label{eq:Q_decomposition_sss}
	Q(0)-Q(x) =\beta_L |x|^{2\alpha}  P_L(x) +\beta_N |x|^{2\alpha} P_N(x),
\end{align}
where we denoted
\begin{equation*}
	P_L(x)= \frac{xx^\top}{|x|^2}, \qquad P_N(x)= I - \frac{xx^\top}{|x|^2},
\end{equation*}
 and $\beta_L,\beta_N>0$ satisfy
\begin{equation}\label{eq:beta_LeN}
	\beta_N=\left(1+ \frac{2\alpha}{d-1}\right)\beta_L,
\end{equation}
enforcing the divergence-free condition. In addition, a direct computation shows
\begin{align} \label{eq:D2G_decomposition_sss}
	\begin{split}
		D^2 G_s (x) 
		&= |x|^{-d -2+ 2s} \big( \gamma_{s,L} P_L(x) + \gamma_{s,N}P_N(x) \big)\\
		&= |x|^{-d -2+ 2s} \big( (\gamma_{s,L}-\gamma_{s,N}) P_L(x)  + \gamma_{s,N}\, I  \big),
	\end{split}
\end{align}
where
\begin{align*}
	&\gamma_{s,L} = c_{d,s}(-d+2s)(-d-1+2s),\\
	&\gamma_{s,N} = c_{d,s}(-d +2s),\\
	&\gamma_{s,L} - \gamma_{s,N} = c_{d,s} (-d+2s)(-d-2+2s).
\end{align*}
We are interested in computing explicitly the term $F_t$ appearing in \eqref{eq:time_evolution_sss}. The computations are simplified thanks to the decomposition in longitudinal and normal components for both $Q$ \eqref{eq:Q_decomposition_sss} and $D^2G_s$ \eqref{eq:D2G_decomposition_sss}. We have
\begin{align}\label{eq:QDG_1}
	\begin{split}
	&(Q(0)-Q) D^2G_s (x)\\
	&\qquad= |x|^{-d-2+2s+2\alpha} \big(  \gamma_{s,L}\beta_L  P_L(x)+ \gamma_{s,N} \beta_N   P_N(x) \big)\\
	&\qquad= |x|^{-d-2+2s+2\alpha}\big(  (\gamma_{s,L}\beta_L - \gamma_{s,N} \beta_N) P_L(x) + \gamma_{s,N} \beta_N    I \big) ,
	\end{split}
\end{align}
hence, 
\begin{equation}\label{eq:QDG_2}
	\trace\big((Q(0)-Q) D^2G_s \big) (x) =  |x|^{-d-2+2s+2\alpha} \big(\gamma_{s,L} \beta_L +(d-1) \gamma_{s,N} \beta_N \big) .
\end{equation}
Secondly, 
\begin{align}\label{eq:QDG_3}
	\begin{split}
	&\big(\trace (Q(0)-Q)  \big) D^2G_s (x)\\
	&\qquad= |x|^{-d-2+2s +2\alpha} (\beta_L + (d-1) \beta_N )\big((\gamma_{s,L}-\gamma_{s,N}) P_L(x)+ \gamma_{s,N} I\big).
	\end{split}
\end{align}
By \eqref{eq:QDG_1}-\eqref{eq:QDG_3},
\begin{equation*}
	F_t 
	= \pi_1 \langle |x|^{-d-2+2s+2\alpha} I \ast \mf, \mf\rangle + \pi_2 \langle |x|^{-d-2+2s+2\alpha} P_L(x) \ast \mf, \mf\rangle,\\
\end{equation*}
where $\pi_1=\pi_1(d,s,\alpha)$ and $\pi_2=\pi_2(d,s,\alpha)$ are given by
\begin{align*}
	&\pi_1=-\gamma_{s,L}\beta_L+(4-2d)\gamma_{s,N}\beta_N - \gamma_{s,N} \beta_L,\\
	&\pi_2=\gamma_{s,L}\beta_L+ (d-3)\gamma_{s,N} \beta_N  -(d-1)\gamma_{s,L}\beta_N  + \gamma_{s,N}\beta_L.
\end{align*}
We define $\delta=\delta(d,s,\alpha)$ by imposing $\pi_2= \delta(\gamma_{s+\alpha,L}-\gamma_{s+\alpha,N})$.
By adding and subtracting $\pm \delta \gamma_{s+\alpha,N} \langle |x|^{-d-2+2s+2\alpha} I \ast \mf, \mf\rangle$ and recalling the decomposition \eqref{eq:D2G_decomposition_sss}, we obtain
\begin{align*}
	F_t & = ( \pi_1 - \delta \gamma_{s+\alpha,N})  \langle |x|^{-d-2+2s+2\alpha} I \ast \mf, \mf\rangle + \delta \langle D^2 G_{s+\alpha} \ast \mf, \mf\rangle\\
	& =  ( \pi_1 - \delta \gamma_{s+\alpha,N})  \langle |x|^{-d-2+2s+2\alpha} I \ast \mf, \mf\rangle\\
	& =  \tilde \eta_{d,s,\alpha}  \|\mf\|^2_{\dot H^{-s+1-\alpha}},
\end{align*}
where we used $\langle D^2 G_{s+\alpha} \ast \mf, \mf\rangle=0$ (by integration by parts and divergence free condition) and we defined
\begin{equation*}
	\tilde \eta_{d,s,\alpha} \coloneqq \frac{\pi_1 - \delta \gamma_{s+\alpha,N}}{c_{d,s+\alpha-1}},
\end{equation*}
provided that $s\in (1-\alpha, d/2)$ and $\alpha\in(0,1)$.
Therefore, taking the expectation in \eqref{eq:time_evolution_sss} we have
\begin{equation}\label{eq:main_formula_sss}
		\dd \E \|\mf_t\|^2_{\dot H^{-s}} + 
		 \tilde \eta_{d,s,\alpha}  \E\|\mf\|^2_{\dot H^{-s+1-\alpha}} \dd t =0,
\end{equation}
showing that the statistical behaviour of the model is determined by the sign of the coefficient $\tilde \eta_{d,s,\alpha}$. Performing some elementary manipulations, we obtain the following formulas
\begin{align*}
	&\pi_1=\beta_L \gamma_{s,N} \left(d-2s + \left ( 1+ \frac{2\alpha}{d-1}\right)(4-2d)\right)\eqcolon \beta_L \gamma_{s,N} \tilde \pi_1,\\
	&\pi_2=\beta_L \gamma_{s,N} \left( -d+2s + \left ( 1+ \frac{2\alpha}{d-1}\right)(d^2+d-2sd+2s-4) \right)\\
    &\quad\eqcolon \beta_L \gamma_{s,N} \tilde \pi_2,
\end{align*}
hence, we can write
\begin{equation*}
	\tilde \eta_{d,s,\alpha} = \frac{\beta_L \gamma_{s,N}}{c_{d,s+\alpha-1}} \left(\tilde \pi_1 - \frac{\tilde \pi_2}{-d-2+2s+2\alpha}\right).
\end{equation*}
For $s\in (1-\alpha, d/2)$ and $\alpha\in(0,1)$, being $\beta_L>0$, $c_{d,s+\alpha-1}>0$ and $\gamma_{s,N}<0$, it follows that $\tilde \eta_{d,s,\alpha}>0$ if and only if 
\begin{equation}\label{eq:condition_positive}
	f (d,s,\alpha)>0,
\end{equation}
where
\begin{align*}
	f (d,s,\alpha)  
	&\coloneqq (-d-2+2s+2\alpha) \tilde \pi_1(d,s,\alpha) - \tilde \pi_2(d,s,\alpha) \\
	&= \frac{-8(d-2) \alpha^2  -8(d-2)\alpha  (s-1) + 2(d-1)(s-1) (d-2s +2) }{d-1}
\end{align*}
One can check that the above expression for $f(d,s,\alpha)$ is, up to a positive multiplicative constant, the expression for the constant $\eta_{d,s,\alpha}$ obtained in the main result \cref{thm:main_bd_integral} (precisely, $f(d,s,\alpha)=2\eta_{d,s,\alpha}/((d-1)C_{d,s,\alpha})$ with $C_{d,s,\alpha}$ as in \cref{thm:main_bd_integral}).

\subsection{Regime of parameters responsible for anomalous regularization}\label{subsec:parameters}
Let us analyze the regime of parameters making $\eta_{d,s,\alpha}>0$ (where $\eta$ is given in \cref{thm:main_bd_integral}), or equivalently $f (d,s,\alpha)>0$.

For $d=2$, $f$ drastically simplifies losing its dependence on the parameter $\alpha$,
\begin{equation*}
	f(2,s,\alpha)= 4(s-1) (2-s),
\end{equation*}
which is always negative for $s\in (0,1)$.
On the other hand, for $d\ge 3$, $f (d,s,\alpha)$ is a quadratic function with negative concavity in both $s$ and $\alpha$. Hence, it is positive if and only if $\alpha$ and $s$ are in between their respective roots. Such a condition identifies an ellipse as illustrated in \cref{fig:ellipsoid}. An explicit computation of the roots of $f$ as function of $s$ shows that
\begin{align*}
	&\spma = \frac d4 + 1 - \alpha \frac{d-2}{d-1} \pm \frac{\sqrt{d}}{4(d-1)} \sqrt{\Delta^s_{d,\alpha}},\\
	&\qquad\Delta^s_{d,\alpha} = -16\alpha^2(d-2)+ d(d-1)^2-8\alpha(d^2-3d+2).
\end{align*}
Clearly, $\Delta^s_{d,\alpha}$ is quadratic function of $\alpha$ with negative concavity and its roots are given by
\begin{equation*}
	\apm = - \frac{d-1}{4} \pm \frac1 4 \sqrt{\frac{2(d-1)^3}{d-2}}.
\end{equation*}
On the other hand, if we compute the roots of $f$ as a function of $\alpha$ we obtain
\begin{equation*}
	\apms = -\frac{s-1}{2} \pm \frac12 \sqrt{\Delta^\alpha_{d,s}},\quad
	\Delta^\alpha_{d,s} = \frac{d(d-s)(s-1)}{d-2},
\end{equation*}
where $\Delta^\alpha_{d,s}$ is a quadratic function of $s$ with negative concavity and its roots are given by
\begin{equation*}
	\sm = 1, \qquad \sp  =d. 
\end{equation*}

\begin{figure}[h]
	\centering
	\begin{tikzpicture}[scale=1.5]
		\begin{scope}
			\clip (0,0) rectangle (2,1.5);
			\fill[orange!50, opacity=0.5] [rotate around={-37.98187826603676:(2.5,-0.75)}] (2.5,-0.75) ellipse (1.7687208533866974cm and 0.8995146151091893cm);
		\end{scope}
		
		\draw[thick,->] (-1,0) -- (5,0) node[right] {$s$};
		\draw[thick,->] (0,-3) -- (0,1.5) node[above] {$\alpha$};
		
		\draw [thick, rotate around={-37.98187826603676:(2.5,-0.75)}] (2.5,-0.75) ellipse (1.7687208533866974cm and 0.8995146151091893cm);

		\fill (0,0) circle (1pt) node[below left] {$O$};
		
		\def\a{0.25}	
		\def\xA{{(6-2*\a - sqrt(9-12*\a-8*(\a)^2))/3}}
		\def\xB{{(6-2*\a + sqrt(9-12*\a-8*(\a)^2))/3}}
		\coordinate (A) at ( \xA,\a);
		\coordinate (B) at ( \xB,\a);
		\fill (A) circle (1pt);
		\fill (B) circle (1pt);
		\draw[semithick ] (A) -- (B);

		\draw[dashed] (-0.3,{3/4*(sqrt(3)-1)}) -- (1.9,{3/4*(sqrt(3)-1)});
		\node[left] at (-0.3,{3/4*(sqrt(3)-1)}) {$\ap$};
		
		\draw[dashed] (-0.3,{3/4*(-sqrt(3)-1)}) -- (4.3,{3/4*(-sqrt(3)-1)});
		\node[left] at (-0.3,{3/4*(-sqrt(3)-1)}) {$\am$};
		
		\draw[dashed] (-0.3,\a) -- (A);
		\node[left] at (-0.3,\a) {$\alpha$};
		
		\draw[dashed] (\xA,-0.1) -- (\xA,1);
		\node[above] at (\xA,1) {$\sma$};
		\draw[dashed] (\xB,-0.1) -- (\xB,1);
		\node[above] at (\xB,1) {$\spa$};
		
		\draw[dashed] (1,-2.3) -- (1,0.08);
		\node[below] at (1,-2.3) {$1$};
		
		\draw[dashed] (4,-2.3) -- (4,0.08);
		\node[below] at (4,-2.3) {$d$};
		
		\draw[dashed, thick] (2,-0.3) -- (2,0.8);
		\node[below] at (2,-0.3) {$\frac d2 $};
		
		\draw[dashed] (3,-0.3) -- (3,0.1);
		\node[below] at (3,-0.3) {$\frac d2+1 $};
	\end{tikzpicture}
	\caption{Ellipse satisfying $\eta_{d,s,\alpha}=0$; in orange the region of parameter providing anomalous regularization.}
	\label{fig:ellipsoid}
\end{figure}
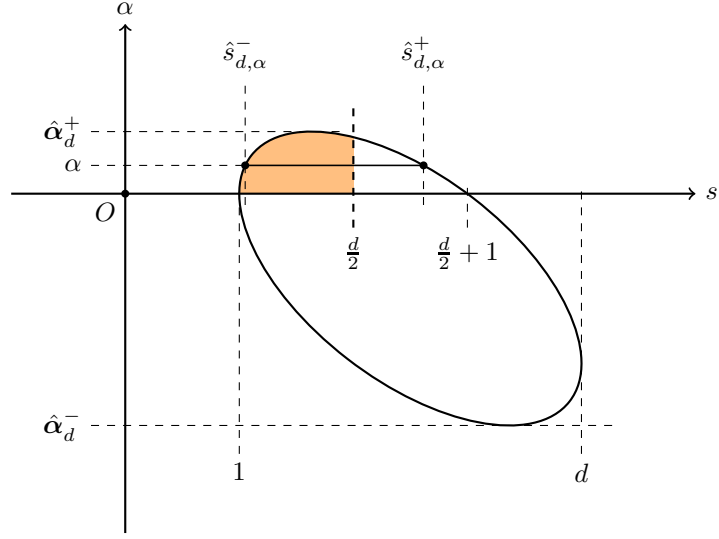

Clearly $\am<0$ and $\ams<0$ (for $s\in (\sm,\sp)$) being the sum of negative terms. In addition, it easy to see that $\hat{\bm{\alpha}}^+_3= 1/2$, $\ap> 1/2$ for $d>3$ and $\ap \uparrow + \infty$ as $d\to +\infty$. Note that $\ap$ is \textit{reached} by $\sma\big|_{\alpha=\ap}=\spa\big|_{\alpha=\ap}<d/2$ for $d\ge 4$ and by $s=3/2$ for $d=3$. Moreover, $\aps>0$ if and only if $s\in(1,d/2+1)$.

Since the dimension $d=3$ plays a crucial role in the applications of the model, we conclude by writing the formula for $\hat s_{d,\alpha}$ in this particular case
\begin{equation*}
	\hat s_{3,\alpha}^\pm = \frac 74 - \frac \alpha2 \pm \frac {\sqrt3}{4} \sqrt{3-4\alpha-4\alpha^2}, \qquad \alpha\in\left(0,\frac 12\right).
\end{equation*}

\subsection{Istantaneous blow-up of the $L^2$ norm}\label{sec:L2_blow_up}
We now provide some additional heuristic computations suggesting the istantaneous blow-up of the $L^2$ norm of $\mf$. We remark that this issue is linked to the difficulty discussed in \cite[Section 6]{FlandoliLuo2021} and the following computations are the $\R^d$ counterpart of what there is presented on the torus $\T^3$. However, a difference is that in \cite{FlandoliLuo2021} the authors are interested in a scaling limit procedure of smooth-in-space noises (this procedure is prevented there by the inability to uniformly control the $L^2$ norm), while here we are somehow already studying a limit equation.

 By applying the It\^o formula, we derive
\begin{multline*}
    \dd \|\mf_t\|^2_{L^2} +2 \sum_k \langle \mf_t, \sigma_k \cdot \nabla \mf_t- \mf_t\cdot \nabla \sigma_k \rangle  \dd W^k_t\\
	=c_0 \langle \mf_t, \Delta\mf_t \rangle \dd t +\sum_{k} \langle \sigma_k \cdot \nabla \mf_t- \mf_t\cdot \nabla \sigma_k, \sigma_k \cdot \nabla \mf_t- \mf_t\cdot \nabla \sigma_k\rangle \dd t.
\end{multline*}
First, note that the term
\begin{equation*}
	c_0 \langle \mf_t, \Delta\mf_t \rangle = - c_0 \|\nabla \mf_t\|^2_{L^2}
\end{equation*}
is perfectly compensated by the term
\begin{align*}
	\sum_k \langle \sigma_k \cdot \nabla \mf_t, \sigma_k \cdot \nabla \mf_t \rangle 
	&= \sum_k \int_{\R^d} \sigma_k^l(x) \partial_l \mf^i_t(x) \sigma_k^m (x) \partial_m \mf^i_t (x) \dd x\\
	&=\int_{\R^d} Q^{lm}(0)\partial_l \mf^i_t(x)  \partial_m \mf^i_t (x) \dd x\\
	&=c_0\int_{\R^d} \delta_{lm}\partial_l \mf^i_t(x)  \partial_m \mf^i_t (x) \dd x= c_0 \|\nabla \mf_t\|^2_{L^2}.
\end{align*}
Then, we have
\begin{align*}
	-2 \sum_k \langle \sigma_k \cdot \nabla \mf_t,\mf_t \cdot \nabla \sigma_k \rangle 
	&= -2 \sum_{k} \int_{\R^d} \sigma_k^l(x) \partial_l \mf^i_t(x) \mf^m_t (x) \partial_m \sigma_k^i (x) \dd x\\
	&= 2 \partial_m Q^{l1}(0) \int_{\R^d}  \partial_l \mf^i_t(x) \mf^m_t (x) \dd x = 0,
\end{align*}
where we used that
\begin{equation*}
	\sum_{k} \sigma_k^l(x) \partial_m \sigma_k^i (x) 
	= \sum_{k} \partial_{y_m} (\sigma_k^l(x) \sigma_k^i(y)) \big|_{y=x} = - \partial_m Q^{li}(0)=0
\end{equation*}
and the last equality is formally justified by $Q$ being an even function (although here not differentiable). Ignoring the martingale part, the blow-up issue comes from the remaining term. We have
\begin{align*}
	\sum_{k} \langle \mf_t \cdot \nabla \sigma_k, \mf_t \cdot \nabla \sigma_k \rangle 
	&= \sum_{k} \int_{\R^d} \mf_t^l(x) \partial_l \sigma_k^i(x) \mf^m_t (x) \partial_m \sigma_k^i (x) \dd x\\
	&= - \partial_l \partial_m \trace Q (0) \int_{\R^d} \mf_t^l(x) \mf^m_t (x)  \dd x,
\end{align*}
where we used
\begin{equation*}
	\sum_{k} \partial_l \sigma_k^i(x) \partial_m \sigma_k^i (x) 
	=\sum_{k} \partial_{x_l} \partial_{y_m} (\sigma_k^i(x)  \sigma_k^i (y) )\big|_{y=x}
	= -\partial_l \partial_m \trace Q(0). 
\end{equation*}
However, being
\begin{equation*}
	-\partial_l \partial_m \trace Q(0) = -i^2 (d-1)\int_{\R^d} k^l k^m \langle k \rangle ^{-d-2\alpha} \dd k,
\end{equation*}
we formally have
\begin{equation*}
	-\partial_l \partial_m \trace Q(0) = \delta_{lm} (d-1) \int_{\R^d} |k|^2 \langle k \rangle ^{-d-2\alpha} \dd k =  \delta_{lm} \cdot (+\infty),
\end{equation*}
which suggests the instantaneous blow-up of the $L^2$ norm of $\mf$.

We observe that the previous computations can also be obtained by formally substituting $s=0$ in \cref{lem:HS_homogeneous}. One can reach the same formula after some manipulations and recalling the divergence-free condition satisfied by $\mf$.

\section{Rigorous Setting and Main Result}\label{sec:Setting_and_Main_Result}

\subsection{Incompressible nonsmooth Kazantsev-Kraichnan noise}
On a filtered probability space $(\Omega,\cA,(\cF_t)_{t\ge0}, \P)$ satisfying the standard assumption, we consider a $\R^d$-valued and $(\cF_t)_t$-adapted spatially homogeneous Wiener process $W^\cK$ (see \cite[Section 4.1.4]{DaPratoZabczyk2014}). We assume that $W^{\cK}$ has covariance function $Q$, i.e., for $x,y \in \R^d$ and $s,t\ge 0$,
\begin{equation*}
	\E[W^\cK(t,x)W^\cK(s,y)^\top] = (t\wedge s) \,Q(x-y),
\end{equation*}
where the function $Q:\R^d \to \R^{d\times d}$ has Fourier spectrum
\begin{equation}\label{eq:Q_spectrum}
	\widehat{Q}(n) = \bracn^{-d-2\alpha} \left(I-\frac{n n^\top}{|n|^2}\right) = \bracn^{-d-2\alpha} P^\perp_k.
\end{equation}
The Kraichnan velocity field $u^\cK$ is formally given by the time derivative of $W^\cK$
\begin{equation*}
	u^\cK \coloneqq \frac{\mathrm{d}W^\cK}{\mathrm{d}t}.
\end{equation*}

In this work, we shall consider $\alpha \in (0,1)$; as mentioned in the introduction, such slow decay in Fourier modes corresponds to a rough (in space) velocity field, precisely $W_t$ belongs to $C^{\alpha-\epsilon}_{loc}$ for every $\eps>0$. The projection onto vectors orthogonal to $n$ is the Fourier representation of the Leray projection, and is responsible for making $W^\cK$ divergence free.

For every $\varphi,\psi \in \cS(\R^d)$, the spatially homogenous Wiener process $W^\cK$ is characterized by
\begin{equation*}
	\E \left[ \langle W^\cK_t,\varphi\rangle\ \langle W^\cK_s,\psi \rangle \right] = (t \wedge s) \, \langle Q \ast \varphi, \psi \rangle = (t \wedge s) \, \langle \cQ \varphi, \psi \rangle,
\end{equation*}
where the covariance operator $\cQ$ is defined by
\begin{equation*}
	\cQ \varphi (x) \coloneqq (Q\ast \varphi)(x) = \int_{\R^d} Q(x-y) \varphi(y) \dd y,
\end{equation*}
that is the Fourier multiplier operator associated with \eqref{eq:Q_spectrum}. Since $\hat{Q}$ is a linear bounded operator on $L^2(\R^d)$, we can regard $W^\cK$ as a generalized Wiener process in $L^2(\R^d)$ (in the sense of \cite[Section 4.1.2]{DaPratoZabczyk2014}). Note that, as a consequence of working on the whole $\R^d$ (instead of $\T^d$), $\cQ$ is not a trace-class operator (see \cref{lem:Q_not_trace_class}) and does not take values in $L^2(\R^d)$, but only in $L^2_\loc(\R^d)$.

It is easy to recognize $\cQ^\frac12$ as the Fourier multiplier operator associated to the multiplier
\begin{equation*}
	m(n) = \bracn^{-\frac d2-\alpha} \left(I-\frac{n n^\top}{|n|^2}\right).
\end{equation*}
Set $U_0\coloneqq \cQ^\frac12 \left( L^2(\R^d) \right)$ to be the so-called reproducing kernel associated to $W^\cK$;  $U_0$ is endowed with the norm $\|u\|_{U_0} \coloneqq \| \cQ^{-\frac12} u \|_{L^2}$, where we denoted by $\cQ^{-\frac 12}$ the pseudo-inverse of $\cQ^{\frac 12}$. The following classical lemma characterize the reproducing kernel of $W^\cK$, for completeness we give a proof in the appendix:
\begin{lemma}\label{lem:rep_kernel}
    We have
    \begin{equation*}
        U_0\coloneqq \cQ^\frac12 \left( L^2(\R^d) \right)=H^{\frac d2 + \alpha}_{\div}
    \end{equation*}
    and 
    \begin{equation*}
        \cQ^{-\frac12} \left(H^{\frac d2 + \alpha}_{\div}\right) = L^2_{\div},
    \end{equation*}
    In particular, $\|u\|_{U_0}=\|u\|_{H^{\frac d2 + \alpha}}=\| \cQ^{-\frac 12 }u\|_{L^2}$.
\end{lemma}

In the sequel, we make use of a (classical) series representation of the noise, as specified by the next lemma:
\begin{lemma}\label{lem:noise_series}
	Given any orthonormal basis $\{e_k\}_{k\ge 1}$ of $L^2_{\div} (\R^d)$,  or, equivalently, any orthonormal basis $\{\sigma_k = \cQ^\frac 12 e_k\}_{k\ge 1}$ of $U_0$, there exists a sequence $\{W^k\}_{k\ge 1}$ of independent one-dimensional $(\cF_t)_{t}$-Brownian motions such that
	\begin{equation*}
		W^\cK_t 
		= \sum_{k\ge1} \sigma_k W^k_t 
		=\sum_{k\ge1} \cQ^\frac 12 e_k W^k_t.
	\end{equation*}
	In particular, if $U_0$ embeds in a Hilbert space $U_1$ with a Hilbert-Schmidt embedding, the series is convergent and defines an $U_1$-valued Wiener process.
\end{lemma} 
\begin{proof}
	See \cite[Proposition 4.7 and 4.8]{DaPratoZabczyk2014}.
\end{proof}

The following lemma links any orthonormal basis $\{\sigma_k\}_{k\ge 1}$ of $H^{\frac d2+\alpha}_{\div}$ with the convolution kernel $Q$.

\begin{lemma} \label{lem:sigma_k}
	Let $\{\sigma_k\}_{k\ge1}$ be any orthonormal basis of $U_0=H^{\frac d2 +\alpha}_{\div}$. Then, for every $x,y \in \R^d$ we have
	\begin{equation*}
		Q(x-y)=\sum_{k\ge1} \sigma_k(x) \otimes \sigma_k(y),
	\end{equation*}
	where the series converges absolutely and uniformly on compact sets, and
	\begin{equation*}
		|Q(x-y)| \le |Q(0)| < \infty.
	\end{equation*}
\end{lemma}
\begin{proof}
	See \cite[Lemma 2.3]{GalLuo2023} and \cite[Lemma 2.5]{coghi2023existence}.
\end{proof} 

The explicit representation of $Q$ can be given similarly to \eqref{eq:Q_decomposition_sss}, see e.g. \cite[Section 10]{LeJRai2002}. In particular, we will use that
\begin{align}
\begin{aligned}\label{eq:Q0}
Q(0) &= c_0I, \\
c_0 &= (2\pi)^{-d/2}\int_{\R^d} \langle k\rangle^{-d-2\alpha} |P^\perp_k v|^2 \dd k
\\
&= (2\pi)^{-d/2}\frac{d-1}{d} \int_0^\infty \rho^{d-1} \langle \rho \rangle^{-d-2\alpha} \dd \rho, 
\end{aligned}
\end{align}
where $I$ is the $d\times d$ identity matrix and $v$ is any vector of $\mathbb{S}^{d-1}$.

The reader can find further mathematical properties of the Kraichnan noise for instance in \cite{GalLuo2023, coghi2023existence} and references therein.

\subsection{The transport plus stretching operator}\label{subsec:transport_stretching}

In this subsection, we introduce the transport plus stretching operator in equation \eqref{eq:KK}, that is we give a rigorous meaning to the term
\begin{align}\label{eq:transport_stretching}
    (u^\cK\cdot\nabla) \mf -(\mf\cdot\nabla) u^\cK
\end{align}
when $u^\cK$ is the \emph{nonsmooth} Kraichnan vector field. We also give some bounds on this term that we will use in the paper.

To give a rigorous definition of \eqref{eq:transport_stretching}, we face two difficulties: the time irregularity of $u^\cK$, which is a white noise, and the space irregularity of both $u^\cK$ and $\mf$ (recall that the solution $\mf$ is only a distribution in space).

Concerning time irregularity, we can deal with it by classical stochastic integration. More precisely, assuming for a moment that $u^\cK$ and $\mf$ are smooth in space, we can understand the transport plus stretching term via Stratonovich integration, namely we define formally
\begin{align*}
    \int_0^t (u^\cK\cdot\nabla) \mf -(\mf\cdot\nabla) u^\cK \dd r &:= \int_0^t (\circ \mathrm{d}W^\cK_s \cdot \nabla \mf_s -  \mf_s \cdot \nabla \circ\mathrm{d}W^\cK_s ) \\
    &= \int_0^t B[\mf_s] \circ \dd W^\cK_s,
\end{align*}
where $\circ$ denotes Stratonovich integration and, for $v$, $w$ divergence-free vector fields,
\begin{equation}\label{eq:def_B}
	B[w]v := (v\cdot \nabla)w -(w\cdot \nabla)v = \div(v\otimes w-w\otimes v),
\end{equation}
so formally \eqref{eq:KK} reads
\begin{equation}\label{eq:KK_Strat}
    \dd M = -B[M] \circ \dd W^\cK.
\end{equation}
Still on a formal level, we can also transform the Stratonovich integral in equation \eqref{eq:KK_Strat} into an It\^o integral plus correction, see \cref{sec:Ito-Stratonovich_correction}, obtaining
\begin{equation}\label{eq:noise_term_ito+lap}
    -\int_0^t B[\mf_s] \circ \dd W^\cK_s = -\int_0^t B[\mf_s] \dd W^\cK_s +\frac{c_0}2 \int_0^t\Delta \mf_s \dd s,
\end{equation}
where $c_0>0$ is such that $Q(0)=c_0I$ (see \eqref{eq:Q0}); here $\frac{c_0}2 \Delta \mf$ represents the It\^o-Stratonovich correction.

The use of the Stratonovich integral is the correct choice here, as generally in transport-type PDEs, because it arises formally in the delta-correlated limit of smooth (in time and space) random velocity fields; see e.g. \cite[Remark 3]{FlaMauNek2014} for a more detailed explanation in this context.
However, as noted in \cite[Section 2.2]{GalLuo2023} for the transport equation, the rigorous definition of the above Stratonovich integral requires some degree of smoothness (in space) of $\mf$ or of $W^\cK$, which we do not have here. Therefore, as in \cite{GalLuo2023}, \cite{coghi2023existence} and \cite{galeati2024anomalous}, to overcome this issue, we use the formulation with the It\^o integral, namely the right-hand side of \eqref{eq:noise_term_ito+lap}, as a rigorous definition of the transport plus stretching term.

Concerning space irregularity of $u^\cK$ and $\mf$ in \eqref{eq:transport_stretching}, in view of the formulation  \eqref{eq:noise_term_ito+lap} with It\^o integral, we need to give sense to the It\^o integral
\begin{equation}\label{eq:noise_term_ito}
	\int_0^t B[\mf_s]\mathrm{d}W^\cK_s,
\end{equation}
where $W^\cK$ is a Wiener process with covariance function $Q$ as in \eqref{eq:Q_spectrum} and $\mf$ takes values in $H^{-s}_{\div}$ for suitable $s>0$. By classical stochastic integration (see \cite[Section 4.2.1]{DaPratoZabczyk2014}), we can define \eqref{eq:noise_term_ito} provided that
$B[\mf]\in L^2\left([0,T];\HS(U_0;H) \right)$ for $\P$-a.s.\ $\omega\in \Omega$, where $U_0=H^{d/2+\alpha}_{\div}$ is the reproducing-kernel Hilbert space of $W^\cK$ (by \cref{lem:rep_kernel}). This is what we will show in the rest of the subsection.

First note that $B[w]v$ can be defined as in \eqref{eq:def_B} for smooth and compactly supported functions $w$ and $v$. We start with extending rigorously the operator $B$ for distributional $w$.

\begin{lemma}\label{lem:products}
	Let $s\in (0,d/2 +\alpha]$, $w \in H^{-s}(\R^d)$ and $v \in H^{d/2+\alpha}(\R^d)$; then, $vw \in H^{-s}(\R^d)$ and 
	\begin{equation*}
		\|vw\|_{H^{-s}} \lesssim \|v \|_{H^{d/2+\alpha}} \| w\|_{H^{-s}}.
	\end{equation*}
\end{lemma}

\begin{proof}
	We resort to paradifferential calculus. In particular, following the notation set in \cite[Section 2.8]{Bahouri2011}, we have the Bony decomposition
	\begin{equation*}
		vw = T_v w + T_w v  + R(v,w).
	\end{equation*}
	Thanks to \cite[Theorem 2.82]{Bahouri2011}, \cite[Theorem 1.66]{Bahouri2011} and \cite[Theorem 2.71]{Bahouri2011}, we have
	\begin{align} \label{eq:PT1}
		&\| T_v w \|_{B^{-s}_{2,2}} 
		\lesssim \|v\|_{L^\infty} \| w \|_{B^{-s}_{2,2}}
		\lesssim \|v\|_{H^{d/2+\alpha}} \| w \|_{B^{-s}_{2,2}},\\
		&\|  T_w v \|_{B^{\alpha-s}_{2,2}} 
		\lesssim \|  T_w v \|_{B^{\alpha-s}_{2,1}}
		\lesssim \| w \|_{B^{-s-d/2}_{\infty,2}} \| v \|_{B^{d/2+\alpha}_{2,2}} 
		\lesssim \| w \|_{B^{-s}_{2,2}} \| v \|_{B^{d/2 +\alpha}_{2,2}}. \label{eq:PT2}
	\end{align}
	In addition, for $s\in(0,d/2+\alpha)$, \cite[Theorem 2.85]{Bahouri2011} implies
	\begin{equation} \label{eq:PRa}
		\| R(v,w) \|_{B^{\alpha-s}_{2,2}} 
		\lesssim \| R(v,w) \|_{B^{d/2+ \alpha-s}_{1,1}} 
		\lesssim \| w \|_{B^{-s}_{2,2}} \| v \|_{B^{d/2+\alpha}_{2,2}},
	\end{equation}
	while, for $s=d/2+\alpha$, using \cite[Theorem 2.79]{Bahouri2011} and \cite[Theorem 2.85]{Bahouri2011}, we have
	\begin{equation} \label{eq:PRb}
		\| R(v,w) \|_{B^{-d/2}_{2,2}}
		\lesssim \| R(v,w) \|_{B^{-d/2}_{2,\infty}}
		\lesssim \| R(v,w) \|_{B^{0}_{1,\infty}}
		\lesssim \| w \|_{B^{-d/2-\alpha}_{2,2}} \| v \|_{B^{d/2+\alpha}_{2,2}}.
	\end{equation}
	Recalling that for every $s \in \R$ we have the equality $B^s_{2,2}=H^s$ (see \cite[Examples p.99]{Bahouri2011}), we can combine \eqref{eq:PT1}, \eqref{eq:PT2} and \eqref{eq:PRa} (or \eqref{eq:PRb}) to get
	\begin{align*}
		\| v w\|_{H^{-s}}
		&\lesssim \| T_v w \|_{H^{-s}} + \| T_w v\|_{H^{\alpha-s}}  + \| R(v,w)\|_{H^{\alpha-s}} \\
		&\lesssim  \| v \|_{H^{d/2+\alpha}} \|  w\|_{H^{-s}},
	\end{align*}
	which concludes the proof.
\end{proof}

\begin{corollary} \label{lem:B_is_well_defined}
	Let $s \in  (0,d/2 +\alpha]$. Then $B[\cdot]$ is in $L(H^{-s}_{\div}, L(H^{d/2+\alpha}_{\div};H^{-s-1}_{\div}))$.
\end{corollary}

\begin{proof}
    \cref{lem:products} implies that $B[\cdot]$ is in $L(H^{-s}_{\div}, L(H^{d/2+\alpha}_{\div};H^{-s-1}))$. Moreover, as well-known, for divergence-free, smooth vector fields $v$ and $w$, we have
    \begin{equation*}
        \div[B[w]v] = \sum_{i,j=1}^d \partial_iv^j\partial_jw^i -\partial_iw^j\partial_jv^i =0,
    \end{equation*}
    so that $B[\cdot]$ is in $L(H^{-s}_{\div}, L(H^{d/2+\alpha}_{\div};H^{-s-1}_{\div}))$.
\end{proof}

The following auxiliary lemma gives a representation formula which will be used to show Hilbert-Schmidt bounds on $B$.

\begin{lemma} \label{lem:smooth_Ito_term}
	Let $\mf\in \cS_{\div}$, let $G \in \cS$ be a (smooth) even convolution kernel and let $\{\sigma_k\}_{k\ge1} \subset \cS_{\div}$ be an orthonormal basis of $H^{d/2+\alpha}_{\div}$. Then
	\begin{multline*}
	    \sum_{k\ge 1} \langle G\ast B[\mf]\sigma_k,B[\mf]\sigma_k\rangle \\
			= 
			- \langle \trace(Q D^2G) \ast \mf, \mf\rangle
			- \langle \big((\trace Q) D^2 G  \big)\ast \mf ,\mf \rangle
			+2 \langle (Q D^2G)\ast \mf, \mf \rangle .
	\end{multline*}
\end{lemma}
\begin{proof}
	Since $G$ is even,
	\begin{align}\label{eq:2nd_order_correction}
		\begin{split}
			\sum_{k\ge 1} &\langle G\ast (\sigma_k\cdot\nabla \mf-\mf\cdot\nabla\sigma_k),(\sigma_k\cdot\nabla \mf-\mf\cdot\nabla\sigma_k)\rangle \\
			&= \sum_{k\ge 1} \langle G\ast (\sigma_k\cdot\nabla \mf),\sigma_k\cdot\nabla \mf\rangle - 2 \sum_{k\ge 1} \langle G\ast (\sigma_k\cdot\nabla \mf),\mf\cdot\nabla\sigma_k\rangle \\
			&\quad+\sum_{k\ge 1} \langle G\ast (\mf\cdot\nabla\sigma_k),\mf\cdot\nabla\sigma_k\rangle.
		\end{split}
	\end{align}
	We study each term of the right-hand-side of \eqref{eq:2nd_order_correction}, starting from the last one coming from the stretching part of the noise. Leveraging the divergence-free condition for $\mf$ and recalling \cref{lem:sigma_k}, we have
	\begin{align}
        \begin{aligned}\label{eq:stretching_term}
		\sum_{k\ge 1}& \langle G\ast (\mf\cdot\nabla\sigma_k),\mf\cdot\nabla\sigma_k\rangle  \\
		&= - \sum_{k\ge 1} \iint_{\R^d \times \R^d} \partial_{l_1}\partial_{l_2} G (x-y) \mf^{l_1} (y) \sigma_k^i (y) \mf^{l_2} (x) \sigma_k^i (x) \dd x \dd y \\
		&= -  \iint_{\R^d \times \R^d} \partial_{l_1}\partial_{l_2} G (x-y) \trace Q(x-y) \mf^{l_1} (y)  \mf^{l_2} (x)  \dd x \dd y \\
		&= \langle G\ast (\mf\cdot\nabla\sigma_k),\mf\cdot\nabla\sigma_k\rangle 
		= - \langle \big((\trace Q) D^2 G  \big)\ast \mf ,\mf \rangle.
        \end{aligned}
	\end{align}
	Then we consider the term coming from the transport part of the noise. We have similarly
	\begin{align}
		\sum_{k\ge 1} &\langle G\ast (\sigma_k\cdot\nabla \mf),\sigma_k\cdot\nabla \mf\rangle \nonumber \\
		&= - \sum_{k\ge 1} \iint_{\R^d \times \R^d} \partial_{l_1}\partial_{l_2} G (x-y) \sigma_k^{l_1} (y)\mf^{i} (y)  \sigma_k^{l_2} (x) \mf^{i} (x)  \dd x \dd y \nonumber \\
		&= - \iint_{\R^d \times \R^d} \partial_{l_1}\partial_{l_2} G (x-y) Q^{l_2 l_1}(x-y) \mf^{i} (y)  \mf^{i} (x)  \dd x \dd y \nonumber \\
		&= - \langle \trace(Q D^2G) \ast \mf, \mf\rangle.\label{eq:transport_term}
	\end{align} 
	Lastly we analyze the term coming from the interaction of the transport and stretching parts. By the divergence-free assumption of both $\mf$ and $\sigma_k$ and the fact that $G$ and $Q$ are even, we obtain
	\begin{align}
		\sum_{k\ge 1}& \langle G\ast (\sigma_k\cdot\nabla \mf),\mf\cdot\nabla\sigma_k\rangle \nonumber \\
		&= - \sum_{k\ge 1} \iint_{\R^d \times \R^d} \partial_{l_1}\partial_{l_2} G (x-y) \sigma_k^{l_1} (y)\mf^{i} (y)   \mf^{l_2} (x) \sigma_k^{i} (x)  \dd x \dd y \nonumber \\
		&= -  \iint_{\R^d \times \R^d} \partial_{l_1}\partial_{l_2} G (x-y) Q^{i l_1}(x-y)\mf^{i} (y)   \mf^{l_2} (x)  \dd x \dd y \nonumber \\
		&= - \langle (Q D^2G)\ast \mf, \mf \rangle.\label{eq:mixed_term}
	\end{align}
	Combining \eqref{eq:stretching_term}, \eqref{eq:transport_term} and \eqref{eq:mixed_term}, the statement is proved.
\end{proof}

The following lemma shows a useful expression for the Hilbert-Schmidt norm of $B$ with respect to a suitable \textit{smooth} kernel $G$. It is convenient to write this Hilbert-Schmidt norm in Fourier modes: for $\mf \in \mathcal{S}_{\div}$,
\begin{align*}
    (2\pi)^{-d/2}\sum_k \langle  G \ast B[\mf]\sigma_k , B[\mf]\sigma_k \rangle
    &= \sum_k \int \widehat G(n) |\mathcal{F} B[\mf]\sigma_k|^2(n) \dd n \\
    &= \sum_k \| \cF B[\mf]\sigma_k \|_{L^2_\gamma}^2 \\
    &= \| \cF B [\mf] \|_{\HS(H^{d/2+\alpha}_{\div}; L^2_\gamma)}^2,
\end{align*}
for $\widehat{G}\in \mathcal{S}$ with $\widehat{G}\ge 0$, where $\gamma(\mathrm{d}n) \coloneqq \widehat G (n) \mathrm{d}n$.

\begin{lemma} \label{lem:FB_HS_norm_in_L2}
	Let $\mf \in H^{-d/2-\alpha}_{\div}$, $\widehat G \in \cS$, $\widehat G \ge 0$ and $\gamma(\mathrm{d}n) \coloneqq \widehat G (n) \mathrm{d}n$. Then,
	\begin{equation}\label{eq:FB_HS_norm_in_L2}
		\| \cF B [\mf] \|_{\HS(H^{d/2+\alpha}_{\div}; L^2_\gamma)}^2 
		= (2\pi)^{-d/2}\int_{\R^d}  \mathbb F_{\gamma} (n) \widehat{\mf} (n) \cdot \overline{\widehat{\mf}} (n) \dd n
	\end{equation}
	where
	\begin{equation}\label{eq:F_gamma}
		\mathbb F_\gamma (n) \coloneqq \int_{\R^d} \langle n-k \rangle^{-d-2\alpha} \widehat{G}(k) \left(|P^\perp_{n-k}k|^2 I +(d-1)\, k k^\top -2P^\perp_{n-k}\, k k^\top \right) \dd k.
	\end{equation}
\end{lemma}
\begin{proof} 
	First we consider $\mf \in \mathcal S_{\div}$.
	Without loss of generality, we can choose an orthonormal basis $\{\sigma_k\}_{k\ge 1 }$ of $H^{\frac d2 + \alpha}_{\div}$ made of elements of $\cS_{\div}$.
	By \cref{lem:smooth_Ito_term}, we obtain
	\begin{align*}
		&(2\pi)^{d/2}\| \cF B [\mf] \|_{\HS(H^{d/2+\alpha}_{\div}; L^2_\gamma)}^2 = \sum_{k\ge 1} \langle  G \ast B[\mf]\sigma_k, B[\mf]\sigma_k\rangle \\
		&= 
		- \langle \trace(Q D^2 G) \ast \mf, \mf\rangle
		- \langle \big((\trace Q) D^2 G  \big)\ast \mf ,\mf \rangle
		+2 \langle (Q D^2 G)\ast \mf, \mf \rangle\\
		&= \int_{\R^d}  \mathbb F_\gamma (n) \widehat{\mf} (n) \cdot \overline{\widehat{\mf}} (n) \dd n,
	\end{align*}
	proving \eqref{eq:F_gamma} for $\mf \in \cS_{\div}$. 
    In view of the general case $\mf \in H^{-d/2-\alpha}_{\div}$, we estimate $\mathbb F_\gamma (n) $, using that $G\in \mathcal{S}$:
	\begin{align} \label{eq:bound_HS_L2}
		\begin{split}
			\left|\mathbb F_\gamma (n)\right|
			& \lesssim \int_{\R^d} \langle n-k \rangle^{-d-2\alpha} \widehat G(k)|k|^{2}  \dd k \\
			& \lesssim \langle n \rangle^{-d-2\alpha} \int_{|k|\le |n|/2} \widehat G(k)|k|^{2}  \dd k \\
                & \quad + \langle n \rangle^{-d-2\alpha} \sup_{|k|>|n|/2}(\langle k \rangle^{d+2+2\alpha}\hat{G}(k)) \int_{|k|>|n|/2} \langle n-k \rangle^{-d-2\alpha}  \dd k  \\
			& \lesssim \langle n \rangle^{-d-2\alpha}.
		\end{split}
	\end{align}
    Hence we have for $\mf \in \cS_{\div}$
	\begin{equation*}
		\| \cF B [\mf] \|_{\HS(H^{d/2+\alpha}_{\div}; L^2_\gamma)}  \lesssim \| \mf \|_{H^{-d/2-\alpha}}.
	\end{equation*} 
    Therefore, for $\mf \in H^{-d/2-\alpha}_{\div}$ (note that $\cF B[\mf]$ is well-defined thanks to \cref{lem:B_is_well_defined}), the representation formula \eqref{eq:F_gamma} follows by a density argument, using \eqref{eq:bound_HS_L2}.
\end{proof}

In the next lemma we bound the Hilbert-Schmidt norm of $B[\mf]$ as operator with values in $\dot{H}^{-s}$.

\begin{lemma}\label{lem:HS_homogeneous}
	Let $s\in (1,d/2+1)$ and $\mf\in H^{-s}_{\div}$. Then, we have
	\begin{equation*}
		\| B[\mf] \|^2_{\HS(H^{d/2 + \alpha}_{\div}; \dot H^{-s})} = (2\pi)^{-d/2} \int_{\R^d}  \mathbb F_{d,s,\alpha} (n) \widehat{\mf} (n) \cdot \overline{\widehat{\mf}} (n) \dd n,
	\end{equation*}
        where
        \begin{equation}\label{eq:F_def}
            \mathbb F_{d,s,\alpha} (n) \coloneqq \int_{\R^d} \langle n-k \rangle^{-d-2\alpha} |k|^{-2s} \left(|P^\perp_{n-k}k|^2 I +(d-1) kk^\top -2P^\perp_{n-k} kk^\top \right) \dd k.
        \end{equation}
	In particular,
	\begin{equation*}
		\| B[\mf] \|_{\operatorname{HS}(H^{d/2 + \alpha}_{\div}; \dot H^{-s})} \lesssim  \|\mf\|_{ H^{-s+1}}.
	\end{equation*}
\end{lemma}
\begin{proof}
	 Let $\widehat G^\delta(n) \in \mathcal S$ be smooth nonnegative, increasing approximations of $ |n|^{-2s}$ (i.e. $0\le \widehat G^\delta(n) \nearrow |n|^{-2s}$). We define the non-negative measures
	 $\gamma(\mathrm{d}n) \coloneqq |n|^{-2s} \lambda (\mathrm{d}n)$ and $\gamma_\delta(\mathrm{d}n) \coloneqq \widehat G^\delta (n) \lambda (\mathrm{d}n)$. Let $\{\sigma_k\}_{k\ge 1 }$ be an orthonormal basis of $H^{\frac d2 + \alpha}_{\div}$.
    We have
	\begin{align*}
	\| B[\mf] \|^2_{\operatorname{HS}(H^{d/2 + \alpha}_{\div}; \dot H^{-s})}
	&= \sum_k \| \cF B[\mf] \sigma_k  \|_{L^2_\gamma}^2 \\
	&= \lim_{\delta\to 0} \| \cF B[\mf] \|^2_{\operatorname{HS}(H^{d/2 + \alpha}_{\div}; L^2_{\gamma_\delta})} 
	\end{align*}
	hence, thanks to \cref{lem:FB_HS_norm_in_L2},
	\begin{equation*}
		\| B[\mf] \|^2_{\operatorname{HS}(H^{d/2 + \alpha}_{\div}; \dot H^{-s})}
		= (2\pi)^{-d/2} \lim_{\delta\to 0} \int_{\R^d}  \mathbb F_{\gamma_\delta} (n) \widehat{\mf} (n) \cdot \overline{\widehat{\mf}} (n) \dd n.
	\end{equation*}
    Moreover $\mathbb F_{\gamma_\delta}$ tends pointwise to $\mathbb F_{\gamma}=\mathbb F_{d,s,\alpha}$, therefore by dominated convergence theorem it is enough to show that $|\mathbb F_{\gamma_\delta}(n)| \lesssim \langle n \rangle^{-2s+2}$ uniformly in $\delta$. We call $s^\prime=s-1$.
    We bound
	\begin{align*}
            |\mathbb F_{\gamma_\delta}(n)| \lesssim \int_{\R^d} \langle n-k \rangle^{-d-2\alpha}  |k| ^{-2s^\prime} \dd k.
	\end{align*}
        For $|n|\le 1$, we have $\langle k\rangle \approx \langle n-k \rangle $ and so $|F_{\gamma_\delta}(n)|\lesssim 1$. For $|n|>1$, we split 
        \begin{align*}
            \int_{\R^d} \langle n-k \rangle^{-d-2\alpha}  |k| ^{-2s^\prime} \dd k
            &=  \int_{|k|\le|n|/2} \ldots +  \int_{|k|>|n|/2} \ldots .
        \end{align*}
	In the region $|k|\le|n|/2$, we have $|n-k|\ge |n|$/2, hence 
	\begin{align*}
		\int_{|k|\le|n|/2} \langle n-k \rangle^{-d-2\alpha}  |k|^{-2s^\prime} \dd k
		&\lesssim \bracn^{-d-2\alpha}\int_{|k|\le|n|/2}  |k|^{-2s^\prime} \dd k \\
		&\lesssim \bracn^{-d-2\alpha} |n|^ {d-2s} \lesssim \bracn^{-2s^\prime} .
	\end{align*}
	In the region $|k|>|n|/2$, we have
	\begin{align*}
		\int_{|k|>|n|/2} \langle n-k \rangle^{-d-2\alpha}  |k|^{-2s^\prime} \dd k
		&\lesssim |n|^{-2s^\prime}\int_{\R^d} \langle n-k \rangle^{-d-2\alpha}   \dd k\\
		&\lesssim |n| ^{-2s^\prime}.
	\end{align*}
        We conclude that, uniformly in $\delta$, for every $n$
        \begin{align*}
            |\mathbb F_{\gamma_\delta}(n)| \lesssim \int_{\R^d} \langle n-k \rangle^{-d-2\alpha}  |k| ^{-2s^\prime} \dd k \lesssim \bracn^{-2s^\prime}.
        \end{align*}
        The proof is complete.        
\end{proof}

By the classical theory of It\^o integration (e.g. \cite[Section 4.2.1]{DaPratoZabczyk2014}), we can make sense and control the It\^o integral \eqref{eq:noise_term_ito}:

\begin{corollary}\label{cor:stoch_int_bd}
	Let $s\in (0,d/2)$. For every $(\cF_t)_t$-progressively measurable process $\mf$ with paths in $L^2_t(H^{-s}_{\div})$ $\P$-a.s., the stochastic integral
	\begin{align*}
		N_t:= \int_0^t B[\mf_r] \dd W^\cK_r
	\end{align*}
	is a well-defined local martingale with values in $\dot{H}^{-s-1}_{\div}$ and the following bound holds for its quadratic variation $[N]$:
	\begin{align}\label{eq:stoch_int_bd} 
		\trace_{\dot{H}^{-s-1}_{\div}}[N]_t  = \int_0^t \| B(\mf_r) \|^2_{\operatorname{HS}(H^{d/2 + \alpha}_{\div}; \dot{H}^{-s-1})} \dd r \lesssim \int_0^t \|\mf_r\|_{H^{-s}}^2 \dd r.
	\end{align}
\end{corollary} 

We remark that a finer analysis of the quadratic variation of the noise will be carried out in the sequel and will be responsible of the main result of this work. 

\subsection{Definition of solution}

Based on the definition of transport and stretching term in \cref{subsec:transport_stretching}, we give the rigorous definition of solutions to \eqref{eq:KK} with $u^\cK$ \emph{nonsmooth} Kraichnan noise. We recall that, for given $\alpha\in (0,1)$, $Q$ is the covariance function given by \eqref{eq:Q_spectrum}.

\begin{definition}\label{def:sol}
    Fix $s\in (0,d/2)$ and $\alpha\in (0,1)$. An $\dot{H}^{-s}$-valued, probabilistically weak solution to \eqref{eq:KK} is a tuple $(\Omega,\cA,(\cF_t)_t,\P,W^\cK,\mf)$, where $(\Omega,\cA,(\cF_t)_t,\P)$ is a filtered probability space (satisfying the standard assumption), $W^\cK$ is a $(\cF_t)_t$-adapted Wiener process with covariance function $Q$ (given in \eqref{eq:Q_spectrum}), and $\mf$ is an $\dot{H}^{-s}_{\div}$-valued, $(\cF_t)_t$-progressively measurable process $\mf$ such that $\mf\in L^2_t(\dot{H}^{-s}_{\div})\cap C_t(H^{-s-2}_{\div})$ $\P$-a.s.\ and
    \begin{align}\label{eq:KK_def}
        \mf_t = \mf_0 -\int_0^t B[\mf_r] \dd W^\cK_r +\int_0^t \frac{c_0}{2}\Delta\mf_r \dd r,\quad \forall t\in[0,T],
    \end{align}
    where the equality makes sense in $H^{-s-2}_{\div}$ thanks to \cref{cor:stoch_int_bd}.
\end{definition}

\subsection{Main result}

We start by introducing the main assumptions on $d$, $s$ and $\alpha$.
\begin{hypothesis}\label{hp:main_hp}
    Let
    \begin{equation*}
        d\ge 3,  \quad\quad \alpha\in (0,\ap\wedge 1),  \quad\quad s\in \left(\sma,\spa \wedge \frac{d}{2}\right),
    \end{equation*}
    where
    \begin{align*}
        &\ap = - \frac{d-1}{4} + \frac1 4 \sqrt{\frac{2(d-1)^3}{d-2}},\\
        &\spma = \frac d4 + 1 - \alpha \frac{d-2}{d-1} \pm \frac{\sqrt{d}}{4(d-1)} \sqrt{\Delta^s_{d,\alpha}},
    \end{align*}
    and
    \begin{equation*}
        \Delta^s_{d,\alpha} = -16\alpha^2(d-2)+ d(d-1)^2-8\alpha(d^2-3d+2).
    \end{equation*}
\end{hypothesis}
As discussed in \cref{subsec:parameters}, these assumptions are equivalent to $\eta_{d,s,\alpha}>0$ (the constant in \cref{thm:main_bd_integral}), constrained to $d\ge 3$, $\alpha\in (0,1)$ and $s\in (1,d/2)$; see also \cref{fig:ellipsoid} for a graphical representation of \cref{hp:main_hp}. The following is our main result.

\begin{theorem}\label{thm:main}
    Let $d$, $s$ and $\alpha$ satisfy \cref{hp:main_hp} and let $Q$ be the covariance function given by \eqref{eq:Q_spectrum}. Then, the following hold.
    \begin{itemize}
    \item \textbf{Strong existence:} For every $\mf_0$ in $\dot{H}^{-s}_{\div}$, for every given filtered probability probability space $(\Omega,\cA,(\cF_t)_t,\P)$ (satisfying the standard assumption) and every given $(\cF_t)_t$-adapted Wiener process $W^\cK$ with covariance function $Q$, there exists a $\dot{H}^{-s}$ solution $\mf$ to the vector advection equation \eqref{eq:KK}, which satisfies in addition $\mf\in L^\infty_t(\dot{H}^{-s}_{\div})$ $\P$-a.s., $\mf\in L^\infty_t(L^2_\omega(\dot{H}^{-s}_{\div})) \cap L^2_{t,\omega}(\dot{H}^{-s+1-\alpha}_{\div})$ and
    \begin{align}\label{eq:reg_gain}
        \sup_{t\in [0,T]}\E[\|\mf_t\|_{\dot{H}^{-s}}^2] +\eta_{d,s,\alpha}\int_0^T \E[\|\mf_r\|_{\dot{H}^{-s+1-\alpha}}^2] \dd r \le e^{\rho_{d,s,\alpha}T} \|\mf_0\|_{\dot{H}^{-s}}^2.
    \end{align}
    for constants $\eta_{d,s,\alpha}>0$, $\rho_{d,s,\alpha}\ge 0$ independent of $\mf_0$ and $\mf$.
    \item \textbf{Pathwise uniqueness:} The solution is unique in the class of $\dot{H}^{-s}$-valued solutions $\mf$ satisfying $\mf\in L^2_{t,\omega}(\dot{H}^{-s+1-\alpha}_{\div})$.
    \end{itemize}
\end{theorem}

\section{Viscous Approximations}\label{sec:Viscous_approximations}

In this section we consider viscous approximations for \eqref{eq:KK}. We show existence and uniqueness of such approximations in the classical framework of SPDEs with monotone coefficients first introduced in \cite{KrylovRozovskii1979} and then refined by many authors, see for instance \cite{LiuRockner2010}.

Let $\nu>0$ be the viscosity coefficient. The viscous approximation of \eqref{eq:KK} is
\begin{equation}\label{eq:SPDE_viscous}
	\dd \mf^\nu + B[\mf^\nu] \dd W^\cK 
	= \left( \nu +\frac{c_0}{2} \right) \Delta \mf^\nu \dd t.
\end{equation}
For $s\in (0,d/2)$, the rigorous definition of $\dot{H}^{-s}$-valued solution can be given analogously as in \cref{def:sol}, so we omit it.

\begin{proposition}\label{prop:wellposed_viscous}
    Let $d$, $s$ and $\alpha$ satisfy \cref{hp:main_hp}. For every $\nu>0$ and for every $\mf_0 \in \dot{H}^{-s}_{\div}$ there exists a unique strong solution $\mf^\nu$ to \eqref{eq:SPDE_viscous} in the class of solutions in $L^2_{t,\omega}(\dot{H}^{-s}_{\vee,\div})$ with sample paths in $C_t(\dot{H}^{-s}_{\div})$. Moreover, the solution $\mf^\nu$ satisfies $\E \sup_{t\in[0,T]} \|\mf^\nu_t\|^2_{\dot H^{-s}}<\infty$ and the following It\^o formula holds $\P$-a.s.:
    \begin{align}
    \begin{aligned}\label{eq:energy_balance}
        \|\mf^\nu_t\|^2_{\dot H^{-s}} 
        &+ (2\nu + c_0) \int_0^t \|\mf^\nu_r\|^2_{\dot H^{-s+1}} \dd r - \int_0^t  \|B[\mf^\nu_r]\|^2_{\HS(H^{d/2+\alpha}_{\div};\dot H^s_{\div})}\dd r \\
        &= \|\mf_0\|^2_{\dot H^{-s}} + 2\int_0^t \langle \mf^\nu_r, B[\mf^\nu_r]  \dd W^\cK_r \rangle_{\dot H^{-s}},\quad \forall t\in [0,T].
    \end{aligned}
    \end{align}
\end{proposition}
With the purpose of proving \cref{prop:wellposed_viscous}, we introduce some auxiliary function spaces. For $s\in\R$, we define the spaces $\dot H^{s+1}_\vee(\R^d)$ and $\dot H^{s-1}_\wedge (\R^d)$ of tempered distributions $f$ over $\R^d$ whose Fourier transform $\widehat f$ belongs to $L^1_{\loc}(\R^d)$ and satisfies respectively
	\begin{align*}
		&\|f\|^2_{\dot H^{s+1}_\vee} \coloneqq \int_{\R^d} |n|^{2s} \vee |n|^{2s+2} |\widehat f (n)|^2 \dd n <\infty,\\
		&\|f\|^2_{\dot H^{s-1}_\wedge} \coloneqq \int_{\R^d} |n|^{2s} \wedge |n|^{2s-2} |\widehat f (n)|^2 \dd n <\infty.
	\end{align*}

We recall the following duality relation, which will be shown in the Appendix:

\begin{lemma}\label{lem:topological_dual}
	For $s<d/2$, $\dot H^{s+1}_\vee(\R^d)$ and $\dot H^{s-1}_\wedge (\R^d)$ are Hilbert spaces. Moreover we can identify the topological dual $(\dot H^{s+1}_\vee)^*$ with $\dot H^{s-1}_\wedge$ via the $\dot H^s$ scalar product.
\end{lemma}

We denote with $\dot H^{s+1}_{\vee, \div}$ and $\dot H^{s-1}_{\wedge, \div}$ the corresponding closed subspaces of tempered distributions satisfying a divergence-free condition. \cref{lem:topological_dual} can be extended to such divergence-free spaces without remarkable differences.

With the aim of using the setting in \cite{LiuRockner2010} and prove \cref{prop:wellposed_viscous}, we introduce the triplet $V\hookrightarrow H\hookrightarrow V^*$, where
\begin{align*}
	H=\dot{H}^{-s}_{\div},\quad\quad V=\dot H^{-s+1}_{\vee, \div}, \quad\quad V^*=\dot H^{-s-1}_{\wedge, \div}.
\end{align*}

We also introduce the operator $A^\nu : V\to V^*$ defined by
\begin{equation}\label{eq:operatorA}
		A^\nu v= \left( \nu +\frac{c_0}{2} \right) \Delta v, 
\end{equation}
and rewrite the viscous approximations \eqref{eq:SPDE_viscous} as
\begin{equation*}
	d\mf^\nu + B[\mf^\nu] \dd W^\cK_t = A^\nu \mf^\nu \dd t.
\end{equation*}
We list here some properties satisfied by $A^\nu$, whose proof is immediate:
\begin{subequations}
	\begin{align}
		&\|A^\nu v \|_{V^*} \lesssim  \| v\|_V, \label{eq:a} \\
		& _{V^*}\langle Av , v\rangle_V = - \left( \nu + \frac{c_0}2\right) \| v \|^2_{\dot H^{-s+1}},\label{eq:b}\\
		&| _{V^*}\langle Av , w\rangle_V | \lesssim \|v\|_V \|w\|_V.\label{eq:c}
	\end{align}
\end{subequations}
\begin{proof} [Proof of \cref{prop:wellposed_viscous}]
	The proof is an application of the classical theory for SPDEs with monotone coefficients \cite{KrylovRozovskii1979}. For convenience we refer to it through \cite[Theorem 4.1]{LiuRockner2010} (we will use the same notation in \cite{LiuRockner2010} except that we use $a$ in place of $\alpha$). The hemicontinuity property (\cite[eq.\ (A1)]{LiuRockner2010}) follows from \eqref{eq:c}. The monotonicity  (\cite[eq.\ (A2)]{LiuRockner2010}) is a consequence of \eqref{eq:b} and the main a priori bound \eqref{eq:d}: indeed, for every $v \in V$ it holds
    \begin{align*}
        2 _{V^*}\langle Av &, v\rangle_V -\|B[v]\|^2_{\HS(H^{d/2+\alpha}_{\div};\dot H^s_{\div})} \\
        &= -(2\nu + c_0) \|v\|_{\dot{H}^{-s+1}}^2 -\|B[v]\|^2_{\HS(H^{d/2+\alpha}_{\div};\dot H^s_{\div})} \\
        &\le -2\nu \|v\|_{\dot{H}^{-s+1}}^2 +\rho_{d,s,\alpha}\|v\|_{\dot{H}^{-s}}^2 \le \rho_{d,s,\alpha}\|v\|_{H}^2.
    \end{align*}
    The coercivity (\cite[eq.\ (A3)]{LiuRockner2010}) holds with $a=2$ and $\theta \in (0,2\nu]$, relying again on \eqref{eq:b}, \eqref{eq:d} and noting that
	\begin{align*}
            -2\nu \|v\|^2_{\dot H^{-s+1}} + \theta \|v\|^2_V 
		&\le -2\nu \|v\|^2_{\dot H^{-s+1}} + \theta \|v\|^2_{\dot H^{-s+1}}
		+ \theta \|v\|^2_{\dot H^{-s}}\\
		&\le  \theta \|v\|^2_{\dot H^{-s}}
		=  \theta \|v\|^2_H.
	\end{align*}
	Lastly, the growth property (\cite[eq.\ (A4)]{LiuRockner2010}) with $a=2$ follows from \eqref{eq:a}.
\end{proof}

\begin{remark}
    In the proof of \cref{prop:wellposed_viscous}, we have used the main a priori bound \eqref{eq:d} (which we will prove independently in the next section). We expect that this is not essential: for fixed $\nu>0$, a finer estimate should allow to control $\sum_k \|B_k\mf\|_{\dot{H}^{-s}}^2$ with the $\dot{H}^{-s+1}$ norm of $\mf$. In particular, the well-posedness of the viscous equation \eqref{eq:SPDE_viscous} could take place in a wider range of $s$ and $\alpha$. Note also that using \cref{lem:HS_homogeneous} in place of the main bound we get well-posedness of \eqref{eq:SPDE_viscous} for sufficiently large $\nu>0$, for any given $s\in (0,d/2)$ and $\alpha\in (0,1)$.
\end{remark}

\section{Regularization Estimate}\label{sec:Main_Bound}

\begin{proposition}\label{prop:main}
    Let $d$, $s$ and $\alpha$ satisfy \cref{hp:main_hp}. There exist constants $\eta_{d,s,\alpha}>0$, $\rho_{d,s,\alpha}\ge 0$, such that for every $\mf\in \dot H^{-s+1}_{\vee, \div}$,
    \begin{multline}\label{eq:d}
        \|B[v]\|^2_{\HS(H^{d/2+\alpha}_{\div};\dot H^{-s}_{\div})} - c_0 \|\mf \|^2_{\dot H^{-s+1}_{\div}} \\
        \le - \eta_{d,s,\alpha} \|\mf\|^2_{\dot H^{-s-\alpha+1}_{\div}} + \rho_{d,s,\alpha}\|\mf\|^2_{\dot H^{-s}_{\div}}.
    \end{multline}
\end{proposition}
\noindent
By \cref{lem:HS_homogeneous}, for every $M\in \dot H^{-s+1}_{\vee, \div}$ the left-hand side of \eqref{eq:d} can be written as:
\begin{multline}\label{eq:identity_H}
\|B[\mf]\|^2_{\HS(H^{d/2+\alpha}_{\div};\dot H^{-s}_{\div})} -c_0 \|\mf\|^2_{\dot H^{-s+1}_{\div}}\\
= (2\pi)^{-d/2} \int_{\R^d}  \mathbb H_{d,s,\alpha} (n) \widehat{\mf} (n) \cdot \overline{\widehat{\mf}} (n) \dd n,
\end{multline}
with (recall the definition of $F_{d,s,\alpha}$ in \eqref{eq:F_def})
\begin{align}
\begin{aligned}\label{eq:H_def}
    \mathbb H_{d,s,\alpha} (n) &= \mathbb F_{d,s,\alpha} (n) -(2\pi)^{d/2}c_0|n|^{-2s+2}I \\
    &= \int_{\R^d} \langle n-k \rangle^{-d-2\alpha} (|k|^{-2s}-|n|^{-2s}) |P^\perp_{n-k}k|^2 I \dd k \\
    &\quad +(d-1)\int_{\R^d} \langle n-k \rangle^{-d-2\alpha} |k|^{-2s} kk^\top \dd k \\
    &\quad -2\int_{\R^d} \langle n-k \rangle^{-d-2\alpha} |k|^{-2s} P^\perp_{n-k} kk^\top \dd k,
\end{aligned}
\end{align}
where we have used the expression \eqref{eq:Q0} for $c_0$ together with $P^\perp_{n-k}k=P^\perp_{n-k}n$. Moreover, for every divergence-free $M\in \dot{H}^{-s}$ we have $n\cdot \hat{M}(n)=0$ for almost every $n$. Therefore \cref{prop:main} follows from the following estimate for the matrix-valued integral $\mathbb H_{d,s,\alpha} (n)$:

\begin{theorem}\label{thm:main_bd_integral}
    Let $d\ge 2$, $s\in (0,d/2)$ and $\alpha\in (0,1)$ with $s+\alpha>1$. There exist $\eta_{d,s,\alpha}\in \R$, $\rho_{d,s,\alpha}\ge 0$ such that the matrix-valued integral $\mathbb H_{d,s,\alpha} (n)$ defined in \eqref{eq:H_def} satisfies for every $n\in\R^d\setminus\{0\}$, for all $v\in \mathbb{S}^{d-1}$ with $v\cdot n=0$,
    \begin{equation}\label{eq:main_bd_integral}
        |v \cdot \mathbb H_{d,s,\alpha} (n) v +\eta_{d,s,\alpha}|n|^{-2s+2-2\alpha}| \le \rho_{d,s,\alpha}|n|^{-2s}.
    \end{equation}
    The constant $\eta_{d,s,\alpha}$ takes the form, for some $C_{d,s,\alpha}>0$,
    \begin{multline*}
        \eta_{d,s,\alpha} = C_{d,s,\alpha} \Big(  (d-1)(s+\alpha-1)(d+2-2s-2\alpha) -\alpha(d-1)(d+2\alpha) +4\alpha (s+\alpha-1) \Big)
    \end{multline*}
    hence $\eta_{d,s,\alpha}$ is strictly positive if and only if \cref{hp:main_hp} on $d$, $s$ and $\alpha$ is met.
\end{theorem}

The rest of this section is devoted to the proof of \cref{thm:main_bd_integral}. For $|n|\le 1$, by the proof of \cref{lem:HS_homogeneous}, we have $|\mathbb F_{d,s,\alpha}(n)|\lesssim 1$, so $|\mathbb H_{d,s,\alpha}(n)|\lesssim |n|^{-2s+2}$ and \eqref{eq:main_bd_integral} follows easily. Hence we focus on the case $|n|>1$.

For $v\in\mathbb{S}^{d-1}$, we write $v \cdot \mathbb H_{d,s,\alpha} (n) v$ as
\begin{align}
    v \cdot \mathbb H_{d,s,\alpha} (n) v
    &= \int_{\R^d} \langle n-k \rangle^{-d-2\alpha} (|k|^{-2s}-|n|^{-2s}) |P^\perp_{n-k}k|^2 \dd k \nonumber\\
    &\quad +(d-1)\int_{\R^d} \langle n-k \rangle^{-d-2\alpha} |k|^{-2s} (v\cdot k)^2 \dd k \nonumber\\
    &\quad -2\int_{\R^d} \langle n-k \rangle^{-d-2\alpha} |k|^{-2s} (v\cdot P^\perp_{n-k} k)(v\cdot k) \dd k \nonumber\\
    &=: \cI_{tra}(n) +(d-1)\cI_{str}(n,v) -2\cI_{mix}(n,v).\label{eq:three_terms}
\end{align}
The integrals $\cI_{str}(n,v)$ and $\cI_{mix}(n,v)$ come respectively from the stretching term and the combination of the transport and stretching terms in the It\^o formula.
Indeed, the term $\cI_{tra}(n)$ due to the transport term $(u^\cK\cdot \nabla)M$ is exactly the same as in the passive scalar case. 
We will deduce estimates for $\cI_{str}(n,v),\cI_{mix}(n,v)$ applying the strategy of \cite[Section 4]{galeati2024anomalous}, from which we recall:

\begin{proposition}
    Take $s\in (0,d/2)$ and $\alpha\in (0,1)$. Then we have, as $|n|\to\infty$,
\begin{align}\label{eq:bd_It}
    \cI_{tra}(n) = c_{tra}|n|^{-2s+2-2\alpha}  +O(|n|^{-2s}),
\end{align}
where\footnote{The constant $c_{tra}$ comes exactly from \cite[Proposition 4.2]{galeati2024anomalous}, as one can check using $\Gamma(z+1)=z\Gamma(z)$; mind that, unlike in \cite{galeati2024anomalous}, here the integral $\cI_{tra}(n)$ is not multiplied by $(2\pi)^{-d/2}$.}
\begin{align*}
    c_{tra}= -(d-1)(s+\alpha-1)(d+2-2s-2\alpha)C_{d,s,\alpha}
\end{align*}
and
\begin{align}\label{eq:constant_in_bd}
    C_{d,s,\alpha} = -\frac{\pi^{d/2}\Gamma(s+\alpha-1)\gamhalf{d-2s+2}\Gamma(-\alpha)}{4\Gamma(s)\gamhalf{d+2+2\alpha}\gamhalf{d+4-2s-2\alpha}}>0.
\end{align}
\end{proposition}

\subsection{Reduction to the scalar case}

An important difference with respect to the transport case in \cite{galeati2024anomalous} is the fact that $\mathbb H_{d,s,\alpha}$ is matrix-valued instead of scalar-valued: $\cI_{str}(n,v)$ and $\cI_{mix}(n,v)$ depend a priori also on $v\in \mathbb S^{d-1}$.

However the isotropy of the covariance matrix $Q$ and of the involved Green kernel ($|n|^{-2s}$ in Fourier) and the condition $n\cdot v=0$ (coming from the divergence-free condition on $M$) imply no dependence on $v \in \mathbb{S}^{d-1}\cap n^\perp$. Indeed, by symmetry of the integrals -- precisely, with the change of variable $k = Rk^\prime$ for $R$ orthogonal $d\times d$ matrix with $n = |n|Re_1$, $v=Re_2$ -- we have, for every $v\in \mathbb{S}^{d-1}$ orthogonal to $n$,
\begin{align*}
    \cI_{str}(n,v) &= \cI_{str}(|n|e_1,e_2)=:\cI_{str}(|n|),\\
    \cI_{mix}(n,v) &= \cI_{mix}(|n|e_1,e_2)=:\cI_{mix}(|n|),
\end{align*}
hence reducing to the case $n=|n|e_1$, $v=e_2$. In the following we call $\lambda=|n|$.

\subsection{Outline of the argument}

We move to polar coordinates and write the integrals $\cI_{str}$ and $\cI_{mix}$ in terms of integrals of the form $\int_0^\infty h(\lambda t)f(t) \dd t$ for suitable $f$ and $h$, with $\lambda=|n|$. The Mellin transform of a locally integrable function
\begin{align*}
    M[f,z] = \int_0^\infty t^{z-1} f(t) \dd t,
\end{align*}
defined for $z\in \mathbb{C}$ for which the integral is absolutely convergent (the \emph{fundamental strip}), satisfies the Parseval formula
\begin{align}
\begin{aligned}\label{eq:Mellin_Res}
\int_0^\infty h(\lambda t)f(t) \dd t = \frac{1}{2\pi i} \int_{r-i\infty}^{r+i\infty} \frac{M[h,z]M[f,1-z]}{\lambda^z} \dd z,
\end{aligned}
\end{align}
provided that $r+i\R$ is contained in the intersection of the fundamental strips of $M[h,\cdot]$ and $M[f,1-\cdot]$ and that the integral on the right-hand side is absolutely convergent. By residue theorem we have
\begin{align*}
    \int_0^\infty h(\lambda t)f(t) \dd t
    &= \sum_{r<\re z< r'} \res\set{-\lambda^{-z}M[h,z]M[f,1-z]} \\
    &\quad +\frac1{2\pi i}\int_{r'-i\infty}^{r'+i\infty} \frac{M[h,z]M[f,1-z]}{\lambda^z} dz,
\end{align*}
for $r'>r$ such that $r'+i\R$ does not contain poles of $M[h,\cdot]$ or $M[f,1-\cdot]$, provided that $M[h,\cdot]$ and $M[f,1-\cdot]$ can be analytically continued to meromorphic functions of $\mathbb{C}$. We will take $r$ and $r'$ so that the poles in $r<\re z< r'$ give the leading order terms in the expansion for large $\lambda=|n|$ and the integral over $r'+i\R$ gives lower order contributions.

\subsection{Mellin transforms and residues}

We recall that the Gamma function $\Gamma$ has only simple poles, located on the non-positive integers, and
\begin{align*}
    \res_{z=-n}\set{\Gamma(z)} = \frac{(-1)^n}{n!},\quad n\in\N .
\end{align*}
We also recall the properties of the Gamma functions
\begin{align*}
    \Gamma(z+1)=z\Gamma(z),\quad \Gamma(n)=(n-1)!,\quad \Gamma(1/2)=\sqrt{\pi},
\end{align*}
which we will apply without further mention, and the Beta integral
\begin{equation}\label{eq:beta}
    \int_0^\pi |\sin \theta|^\gamma |\cos \theta|^\eta \dd \theta=
    \frac{\gamhalf{\gamma+1}\gamhalf{\eta+1}}{\gamhalf{\gamma+\eta+2}}.
\end{equation} 
The Beta function will appear in our computations as a Mellin transform:

\begin{lemma}
    The Mellin transform of the function
    \begin{align*}
        h(t) = \frac{1}{(1+t^2)^{a}}, \quad a>0,
    \end{align*}
    is absolutely convergent for $0< \re(z)<a$ and it is given by
    \begin{align}\label{eq:mellinbeta}
        M[h,z] = \frac{\gamhalf{z}\gamhalf{a-z}}{2\gamhalf{a}}.
    \end{align}
    The analytic continuation of $M[h,z]$ is a meromorphic function with only simple poles, and
    \begin{equation*}
        \res_{z=a}\set{M[h,z]}=-1.
    \end{equation*}
\end{lemma}

Finally, we recall the following explicit representation of a Mellin transform derived in \cite[Lemma 4.1]{galeati2024anomalous}.

\begin{lemma}\label{lem:integral}
    The Mellin transform of the function
    \begin{align*}
        f_{a,b,s}(t)=t^a \int_0^\pi \frac{(\sin\theta)^b }{(1-2t \cos \theta+t^2)^s} \dd \theta,
        \quad a\in\R,\,b,s>0,
    \end{align*}
    is absolutely convergent for $-a<\re z<2s-a$ and it is given by
    \begin{align}\label{eq:mellinf}
        M[f_{a,b,s},z] =\frac{\sqrt\pi \gamhalf{b-2s+2}\gamhalf{b+1}}{2\Gamma(s)}
        \cdot \frac{\gamhalf{2s-z-a}\gamhalf{z+a}}{\gamhalf{b-a-z+2}\gamhalf{b+a-2s+2+z}}.
    \end{align}
    The analytic continuation of $M[f_{a,b,s},z]$ is a meromorphic function with only simple poles, which are contained in the set
    \begin{align*}
        (2s-a+2\mathbb{N}) \cup (-a-2\mathbb{N})
    \end{align*}
    of zeros of the denominator of \eqref{eq:mellinf}.
\end{lemma}

\subsection{Explicit computations}

By the changes of variable $n-k=\tilde{k}$, $n=\lambda e_1$ and $k=\lambda w$, we write $\cI_{str}$ and $\cI_{mix}$ as
\begin{align*}
    \cI_{str}(\lambda)
    &=\int_{\R^d} \langle \tilde{k} \rangle^{-d-2\alpha} |\lambda e_1-\tilde{k}|^{-2s} \tilde{k}_2^2 \dd \tilde{k} \\
    &= \lambda^{d+2-2s} \int_{\R^d} (1+\lambda^2|w|^2)^{-d/2-\alpha} |e_1-w|^{-2s} w_2^2 \dd w, \\
    \cI_{mix}(\lambda)
    &=\int_{\R^d} \langle \tilde{k} \rangle^{-d-2\alpha} |\lambda e_1-\tilde{k}|^{-2s} (e_2\cdot \lambda P^\perp_{\tilde{k}} e_1)((\lambda e_1-\tilde{k})\cdot e_2) \dd \tilde{k} \\
    &= \lambda^{d+2-2s} \int_{\R^d} (1+\lambda^2|w|^2)^{-d/2-\alpha} |e_1-w|^{-2s} \frac{w_1w_2^2}{|w|^2} \dd w.
\end{align*}
We now set $h(t)=(1+t^2)^{-d-2\alpha}$.

The volume of the $(d-3)$-dimensional sphere is $\omega_{d-3}=2\pi^{d/2-1}/\Gamma(d/2-1)$.
For $d\ge 3$, moving to polar coordinates and using \eqref{eq:beta}, we have
\begin{align*}
    \cI_{str}(\lambda)
    &= \omega_{d-3} \lambda^{d + 2 - 2 s} \int_0^{\infty} \dd r \, r^{d + 1}  h (\lambda r) \int_0^{\pi} \dd \theta_1
    \frac{(\sin \theta_1)^d}{(1 + r^2 - 2 r \cos \theta_1)^s}  \\
    &\quad \cdot \int_0^{\pi}\dd \theta_2 (\sin \theta_2)^{d - 3} (\cos \theta_2)^2 \\
    &= \frac{2\pi^{d/2-1}\gamhalf{3}}{\gamhalf{d+1}}\lambda^{d+2-2s} \int_0^{\infty}  \dd r \, h (\lambda r) f_{d+1,d,s}(r),
\end{align*}
the last line using the notation introduced in \cref{lem:integral}.
Similarly,
\begin{align*}
    \cI_{mix}(\lambda)
    &= \omega_{d-3} \lambda^{d + 2 - 2 s} \int_0^{\infty} \dd r \, r^d  h (\lambda r) \\
    &\quad \cdot \int_0^{\pi} \dd \theta_2 \int_0^{\pi} \dd \theta_1\frac{(\cos \theta_2)^2 (\sin
    \theta_1)^2 \cos \theta_1}{(1 + r^2 - 2 r \cos \theta_1)^s} \sin
    (\theta_1)^{d - 2} (\sin \theta_2)^{d - 3}  \\
    &= \omega_{d - 3} \lambda^{d + 2 - 2 s} \int_0^{\infty} \dd r \, r^d  h (\lambda r) \int_0^{\pi}\dd \theta_1 \frac{\cos \theta_1 (\sin \theta_1)^d}{(1 + r^2 -
    2 r \cos \theta_1)^s}  \\
    &\quad \cdot \int_0^{\pi}\dd \theta_2 (\sin \theta_2)^{d - 3}
    (\cos \theta_2)^2 \\
    &= \frac{2\pi^{d/2-1}\gamhalf{3}}{\gamhalf{d+1}}\lambda^{d+2-2s} \int_0^{\infty} \dd r \, h (\lambda r) f(r)
\end{align*}
where
\begin{align*}
    f (r) &= r^d \int_0^{\pi} \frac{\cos \theta_1 (\sin \theta_1)^d}{(1 +
    r^2 - 2 r \cos \theta_1)^s} d \theta_1\\
    & = \frac{r^d}{d + 1} \left[ \frac{(\sin \theta)^{d + 1}}{(1 + r^2 - 2 r
    \cos \theta)^s} \right]_0^{\pi}
    + \frac{2 s r^{d + 1}}{d + 1}  \int_0^{\pi}
    \frac{(\sin \theta_1)^{d + 2}}{(1 + r^2 - 2 r \cos \theta_1)^{s + 1}} d
    \theta_1\\
    & = \frac{2s}{d+1} f_{d+1,d+2,s+1}(r).
\end{align*}
With the same procedure, we get the same expressions for $\cI_{str}$ and $\cI_{mix}$ also in $d=2$. We can now apply formula \eqref{eq:Mellin_Res}. For the product $h(\lambda r) f_{d+1,d,s}(r)$ in $\cI_{str}$, we take $r$, $r'$ such that $d+2-2s<r<d+2\alpha$ (which is possible if $s+\alpha>1$), $d+2<r'<d+2\alpha+2$. In the strip $r<\re z<r'$, $M[h,\cdot]$ has a single simple pole in $z=d+2\alpha$, with residue $-1$, while $M[f_{d+1,d,s},1-\cdot]$ has a single pole in $z=d+2$, so that the sum of residues in \eqref{eq:Mellin_Res} becomes
\begin{multline}
\label{eq:res_Is}
    \sum_{r<\re z< r'} \res\set{-\lambda^{-z}M[h,z]M[f_{d+1,d,s},1-z]} \\
    = \lambda^{-d-2\alpha}M[f_{d+1,d,s},1-d-2\alpha] +O(\lambda^{-d-2}).
\end{multline}
To bound the integral on $r'+i\R$ in \eqref{eq:Mellin_Res}, we recall the limit
\begin{align*}
    \lim_{|y|\to \infty} |\Gamma(x+iy)|e^{\pi|y|/2}|y|^{1/2-x} = \sqrt{2\pi}.
\end{align*}
Since there are no poles on $r'+i\R$, we have 
\begin{align*}
    |M[h,r'+iy]|\lesssim (1\vee |y|)^p e^{-\pi|y|/2}, \quad  |M[f_{d+1,d,s},1-r'-iy]|\lesssim (1\vee |y|)^q,
\end{align*}
for suitable exponents $p$ and $q$ depending on $d,s,\alpha$.
Therefore
\begin{align}\label{eq:bound_rem_Is}
    \int_{r'-i\infty}^{r'+i\infty} \frac{|M[h,z]M[f_{d+1,d,s},1-z]|}{|\lambda^z|} dz \lesssim \lambda^{-r'} \lesssim \lambda^{-d-2}.
\end{align}
Hence, by formula \eqref{eq:Mellin_Res}, the equality \eqref{eq:res_Is} and the bound \eqref{eq:bound_rem_Is} we get for $\cI_{str}$
\begin{align}\label{eq:bd_Is}
    \cI_{str} &= c_{str}\lambda^{-2s+2-2\alpha} +O(\lambda^{-2s}),
\end{align}
where the constant $c_{str}$ is given by
\begin{align*}
    c_{str}&= \frac{2\pi^{d/2-1}\gamhalf{3}}{\gamhalf{d+1}} M[f_{d+1,d,s},1-d-2\alpha] \\
    &= \frac{2\pi^{d/2-1}\gamhalf{3}}{\gamhalf{d+1}} \cdot \frac{\sqrt\pi \gamhalf{d-2s+2}\gamhalf{d+1}}{2\Gamma(s)} \cdot \frac{\gamhalf{2s+2\alpha-2}\gamhalf{2-2\alpha}}{\gamhalf{d+2\alpha}\gamhalf{d+4-2s-2\alpha}} \\
    &= \frac{\pi^{d/2}\gamhalf{d-2s+2}}{2\Gamma(s)} \cdot \frac{\Gamma(s+\alpha-1)\Gamma(1-\alpha)}{\gamhalf{d+2\alpha}\gamhalf{d+4-2s-2\alpha}} \\
    &= -\alpha(d+2\alpha) \frac{\pi^{d/2}\gamhalf{d-2s+2}}{4\Gamma(s)} \cdot \frac{\Gamma(s+\alpha-1)\Gamma(\alpha)}{\gamhalf{d+2+2\alpha}\gamhalf{d+4-2s-2\alpha}} \\
    &= \alpha(d+2\alpha)C_{d,s,\alpha},
\end{align*}
wth $C_{d,s,\alpha}$ as in \eqref{eq:constant_in_bd}. 

We proceed similarly for the product $h(\lambda r) f_{d+1,d+2,s+1}(r)$ in $\cI_{mix}$. We take $r$, $r'$ such that $d-2s<r<d$, $d+2<r'<d+2\alpha+2$.
In the a simple pole in $z=d+2\alpha$, with residue $-1$, while $M[f_{d+1,d+2,s+1},1-\cdot]$ has one pole in $z=d+2$, so that the sum of residues in \eqref{eq:Mellin_Res} becomes
\begin{align}
\begin{aligned}\label{eq:res_Im}
    \sum_{r<\re z< r'} &\res\set{-\lambda^{-z}M[h,z]M[f_{d+1,d+2,s+1},1-z]} \\
    &= \lambda^{-d-2\alpha}M[f_{d+1,d+2,s+1},1-d-2\alpha] +O(\lambda^{-d-2}).
\end{aligned}
\end{align}
Proceeding as for $\cI_{str}$, we get the same bound on the integral on $r'+i\R$, namely
\begin{align}\label{eq:bound_rem_Im}
    \int_{r'-i\infty}^{r'+i\infty} \frac{|M[h,z]M[f_{d+1,d+2,s+1},1-z]|}{|\lambda^z|} dz \lesssim \lambda^{-r'} \lesssim \lambda^{-d-2}.
\end{align}
Hence, by formula \eqref{eq:Mellin_Res}, the equality \eqref{eq:res_Im} and the bound \eqref{eq:bound_rem_Im} we get for $\cI_{str}$
\begin{align}\label{eq:bd_Im}
    \cI_{mix} &= c_{mix}\lambda^{-2s+2-2\alpha} +O(\lambda^{-2s})
\end{align}
where the constant $c_{mix}$ satisfies
\begin{align*}
    c_{mix} &= \frac{2\pi^{d/2-1}\gamhalf{3}}{\gamhalf{d+1}} \cdot \frac{2s}{d+1} M[f_{d+1,d+2,s-1},1-d-2\alpha] \\
    &= \frac{2\pi^{d/2-1}\gamhalf{3}}{\gamhalf{d+1}} \cdot \frac{2s}{d+1} \cdot \frac{\sqrt\pi \gamhalf{d-2s+2}\gamhalf{d+3}}{2\Gamma(s+1)} \cdot \frac{\gamhalf{2s+2\alpha}\gamhalf{2-2\alpha}}{\gamhalf{d+2+2\alpha}\gamhalf{d-2s-2\alpha+4}} \\
    &= \frac{\pi^{d/2} s \gamhalf{d-2s+2}}{2\Gamma(s+1)} \cdot \frac{\Gamma(s+\alpha)\Gamma(1-\alpha)}{\gamhalf{d+2+2\alpha}\gamhalf{d-2s-2\alpha+4}} \\
    &= -2\alpha(s+\alpha-1) \frac{\pi^{d/2}\gamhalf{d-2s+2}}{4\Gamma(s)} \cdot \frac{\Gamma(s+\alpha-1)\Gamma(-\alpha)}{\gamhalf{d+2+2\alpha}\gamhalf{d-2s-2\alpha+4}} \\
    &= 2\alpha(s+\alpha-1) C_{d,s,\alpha},
\end{align*}
where $C_{d,s,\alpha}$ is given by \eqref{eq:constant_in_bd}. Putting together \eqref{eq:three_terms}, \eqref{eq:bd_It}, \eqref{eq:bd_Is} and \eqref{eq:bd_Im}, we arrive at
\begin{align*}
    v \cdot \mathbb H_{d,s,\alpha} (n) v &= -\eta_{d,s,\alpha}|n|^{-2s+2-2\alpha} +O(|n|^{-2s})
\end{align*}
where
\begin{align*}
    \eta_{d,s,\alpha} &= -c_{tra}-(d-1)c_{str}+2c_{mix} \\
    &= C_{d,s,\alpha} \Big((d-1)(s+\alpha-1)(d+2-2s-2\alpha) \\
    &\qquad \qquad \qquad \qquad \qquad-(d-1)\alpha(d+2\alpha) +4\alpha(s+\alpha-1)\Big).
\end{align*}
The proof of \cref{thm:main_bd_integral} is complete.

\section{Proof of Existence}\label{sec:Proof_Existence}

In this Section, we use the main a priori bound \eqref{eq:d} in a (standard) compactness and convergence argument, to show weak existence of a solution to \eqref{eq:KK} (in the sense of \cref{def:sol}), satisfying also the bound \eqref{eq:reg_gain}. In this Section, we always assume that $d$, $s$ and $\alpha$ satisfy \cref{hp:main_hp}.

\subsection{Tightness}

We show tightness of the family $(\mf^\nu)_\nu$ in a suitable weighted Sobolev space, where for $\nu>0$, $\mf^\nu$ is the solution to the viscous equation \eqref{eq:SPDE_viscous}.

The following lemma gives uniform bounds in  $L^\infty_t(L^2_\omega(\dot{H}^{-s}))$, $L^2_{t,\omega}(\dot{H}^{-s+1-\alpha})$ and $L^{2p}_\omega(L^\infty_t(\dot{H}^{-s}))$, for $p\in (0,1)$:

 \begin{lemma}\label{lem:sup_by_stoch_Gronwall}
    Let $p\in(0,1)$ and $T>0$. There exists $c_p>0$ such that the unique solution $\mf^\nu_t$ of \eqref{eq:SPDE_viscous} starting from $\mf_0 \in \dot H^{-s}$ satisfies
 	\begin{align}
 		\begin{split}\label{eq:reg_gain_nu}
 		&\sup_{\nu \in (0,1]} \Biggl(\sup_{t\in[0,T]} \E \|\mf^\nu_t\|^2_{\dot H^{-s}}
 		+2\nu \E \int_0^T \|\mf^\nu_r\|^2_{\dot H^{-s+1}}\dd r\\
 		&\qquad\qquad+\eta_{d,s,\alpha} \E \int_0^T \|\mf^\nu_r\|^2_{\dot H^{-s-\alpha+1}}\dd r \Biggr) \le \|\mf_0\|^2_{\dot{H}^{-s}}e^{\rho_{d,s,\alpha}T},
 		\end{split}
 	\end{align}
    and
 	\begin{equation}\label{eq:sup_time_nu}
 		 \sup_{\nu \in (0,1]}\E \left(\sup_{t\in[0,T]} \|\mf^\nu_t\|^{2p}_{\dot H^{-s}}\right)
 		\le c_p \|\mf_0\|^{2p}_{\dot{H}^{-s}}  e^{p\rho_{d,s,\alpha}T}.
 	\end{equation}
 \end{lemma}
 \begin{proof}
 	Combining \eqref{eq:energy_balance} with estimate \eqref{eq:d}, we infer
 	\begin{align*}
 		&\|\mf^\nu_t\|^2_{\dot H^{-s}} 
 		+ 2\nu \int_0^t \|\mf^\nu_r\|^2_{\dot H^{-s+1}} \dd r
 		+ \eta_{d,s,\alpha}  \int_0^t \|\mf^\nu_r\|^2_{\dot H^{-s-\alpha+1}}  \dd r \\
 		&\quad\le \|\mf_0\|^2_{\dot H^{-s}} 
 		+ \rho_{d,s,\alpha} \int_0^t \|\mf^\nu_r\|^2_{\dot H^{-s}} \dd r
 		+ 2\int_0^t \langle \mf^\nu_r, B[\mf^\nu_r] \dd W^\cK_r \rangle_{\dot H^{-s}}.
 	\end{align*}
 	Defining $Z(t)=\|\mf^\nu_t\|^2_{\dot H^{-s}} 
 	+ 2\nu \int_0^t \|\mf^\nu_r\|^2_{\dot H^{-s+1}} \dd r
 	+ \eta_{d,s,\alpha}  \int_0^t \|\mf^\nu_r\|^2_{\dot H^{-s-\alpha+1}}  \dd r$, then $Z$ satisfies the following relation
 	\begin{equation*}
 		Z(t) \le Z(0) + \rho_{d,s,\alpha} \int_0^t Z(s) \dd s + N(t),
 	\end{equation*}
 	where $N(t)$ is a continuous local martingale. The statement is then a consequence of the stochastic Gr\"onwall lemma \cite[Theorem 4]{Scheutzow2013}.
 \end{proof}

For $R>0$, we introduce the stopping time and stopped process, 
\begin{align*}
	&\tau_R^\nu \coloneqq \inf \{t \ge 0 \; | \; \|\mf^\nu_t\|_{\dot H^{-s}} \ge R \} \wedge T,\\
	&\mf^{\nu,R}_t \coloneqq \mf^\nu_{t \wedge \tau^\nu_R}.
\end{align*}
For this stopped process, we show the following uniform bound in expectation.

\begin{lemma}\label{lem:stopped_Holder_continuity}
	Fix $\gamma \in \left(0, \frac 12\right)$, $p\ge1$. Then, for every $R>0$ it holds
	\begin{equation*}
		\sup_{\nu \in (0,1]} \E \left[ \|\mf^{\nu,R} \|^p_{C^\gamma([0,T];H^{-s-2})} \right] \lesssim_{\gamma,p, T} R^p
	\end{equation*}
\end{lemma}

\begin{proof}
    We write the stopped version of \eqref{eq:SPDE_viscous} in integral form in the interval $[s,t] \subseteq [0,T]$ and take the expectation of the $p$-th power of the $H^{-s-2}$ norm:
    \begin{align*}
        \E \| \mf^{\nu,R}_t - \mf^{\nu,R}_s\|^p_{H^{-s-2}}
        &\lesssim_p
        \E\left\| \int_s^t B[\mf^{\nu,R}_r] \dd W^\cK_r \right\|^p_{H^{-s-2}} \\
        &+ \E\left\| \left( \nu + \frac{c_0}{2} \right) \int_s^t \Delta \mf^{\nu,R}_r\dd r \right\|^p_{H^{-s-2}}.
    \end{align*}
    By Burkholder-Davis-Gundy inequality and \cref{lem:HS_homogeneous}, we get
	\begin{align*}
		\E\left\| \int_s^t B[\mf^{\nu,R}_r]  \dd W^\cK_r \right\|^p_{H^{-s-2}} 
		&\lesssim_{p,T} \E \left( \int_s^t \| B[\mf^{\nu,R}_r]\|^2_{\HS(H^{d/2 + \alpha}_{\div};H^{-s-2})} \dd r\right)^{\frac p2} \\
		& \lesssim \E \left( \int_s^t \| \mf^{\nu,R}_r\|^2_{H^{-s}} \dd r\right)^{\frac p2} \\
		& \le |t-s|^\frac p2 R^p.
	\end{align*}
    We also have
	\begin{align*}
		\E\left\| \left( \nu + \frac{c_0}{2} \right) \int_s^t \Delta \mf^{\nu,R}_r\dd r \right\|^p_{H^{-s-2}}
		& \lesssim \E \left( \int_s^t \left\|  \mf^{\nu,R}_r\right\|_{H^{-s}} \dd r \right)^p \\
		& \lesssim_T |t-s|^\frac p2 R^p.
	\end{align*}
	Combining the two estimates, we conclude that
	\begin{equation*}
		\E \| \mf^{\nu,R}_t - \mf^{\nu,R}_s\|^p_{H^{-s-2}}  \lesssim_{p,T} |t-s|^\frac p2 R^p.
	\end{equation*}
	As a consequence, by \cite[Theorem B.1.5]{da1996ergodicity}, for $\gamma \in \left(0, \frac 12\right)$, $p> \frac{2}{1-2\gamma}$ and $\beta \coloneqq \frac12 \left(\frac12 + \gamma +\frac1p \right) \in \left(0,\frac12\right)$ (in particular satisfying $\gamma<\beta - 1/p$), we obtain 
    \begin{align*}
		\E \left[\mf^{\nu,R} \right]^p_{C^\gamma([0,T];H^{-s-2})}
        &:= \E \sup_{0\le s<t\le T}\frac{\|\mf^{\nu,R}_t-\mf^{\nu,R}_s\|_{H^{-s-2}}^p}{|t-s|^{\gamma p}} \\
        &\lesssim_{\gamma,p} \E \int_0^T\int_0^T \frac{\| \mf^{\nu,R}_t - \mf^{\nu,R}_s\|^p_{H^{-s-2}}}{|t-s|^{1+\beta p}}\dd s  \dd t \\
		& \lesssim_{\gamma,p, T} R^p \int_0^T\int_0^T |t-s|^{p/2 -1 -\beta p}\dd s  \dd t \lesssim_{\gamma,p, T} R^p.
    \end{align*}
    This estimate, together with the easy bound $\| \mf^{\nu,R}_t \|_{H^{-s-2}}\le \| \mf^{\nu,R}_t \|_{\dot H^{-s}} \le R$, concludes the proof for $p> \frac{2}{1-2\gamma}$. The general case $p\ge1$ follows by Jensen's inequality.
\end{proof}

A corollary of the previous lemma is the uniform bound in probability for $\mf^\nu$ in the space $C_t^\gamma(\dot{H}^{-s-2}))$.

\begin{corollary} \label{cor:uniform_Holder_continuity}
	Fix $\gamma \in \left(0,\frac 12\right)$, $T>0$. For every $\delta>0$, there exists $C_{\delta,\gamma,T}> 0$ such that
	\begin{equation*}
		\sup_{\nu \in (0,1]}\P \left( \|\mf^{\nu} \|_{C^\gamma([0,T];H^{-s-2})} > C_{\delta,\gamma,T} \right) \le \delta.
	\end{equation*}
\end{corollary}

\begin{proof}
	For notational convenience, we call $A^\nu_\delta$ the event
	\begin{equation*}
		A^\nu_\delta \coloneqq \left\{ \|\mf^{\nu} \|_{C^\gamma([0,T];H^{-s-2})} > C_{\delta,\gamma,T} \right\},
	\end{equation*}
	for $C_{\delta,\gamma,T}$ to be determined. We write
	\begin{align} \label{eq:splitted_prob}
		\begin{split}
		\P \left( A^\nu_\delta\right) 
		&= \P\left( A^\nu_\delta \, \cap \{\tau^\nu_R \ge T\}\right) + \P\left( A^\nu_\delta \, \cap \{\tau^\nu_R <T\}\right)\\
		&\le \P\left( \|\mf^{\nu,R} \|_{C^\gamma([0,T];H^{-s-2})} > C_{\delta,\gamma,T}\right) + \P\left( \tau^\nu_R <T\right).
		\end{split}
	\end{align}
    For fixed $p\in(0,1)$, by \cref{lem:sup_by_stoch_Gronwall} we have for some $C_{p,T}>0$
	\begin{align}
        \begin{split}\label{eq:bdd_stopping_time}
		\sup_{\nu\in (0,1]}\P \left( \tau^\nu_R < T \right)
		&\le  \sup_{\nu\in (0,1]}\P \left( \sup_{t\in[0,T]} \|\mf^\nu \|_{\dot H^{-s}}\ge R \right) \\
		&\le R^{-2p} \, \sup_{\nu\in (0,1]}\E\sup_{t\in[0,T]} \|\mf^\nu_t\|^{2p}_{\dot H^{-s}} \le C_{p,T} R^{-2p}.
        \end{split}
	\end{align}
	Moreover, by \cref{lem:stopped_Holder_continuity} we have for some $\hat C_{\gamma,p, T} >0$
	\begin{align}\label{eq:second_prob_bound}
		\begin{split}
		\sup_{\nu \in (0,1]} &\P\left( \|\mf^{\nu,R} \|_{C^\gamma([0,T];H^{-s-2})} > C_{\delta,\gamma,T}\right) \\
		&\le C_{\delta,\gamma,T}^{-1}
		\sup_{\nu \in (0,1]} \E \left[ \|\mf^{\nu,R} \|_{C^\gamma([0,T];H^{-s-2})} \right]\\
		&\le C_{\delta,\gamma,T}^{-1} \hat C_{\gamma,p, T} R.
		\end{split}
	\end{align}
	Inserting \eqref{eq:bdd_stopping_time} and \eqref{eq:second_prob_bound} into \eqref{eq:splitted_prob}, we obtain
	\begin{equation*}
		\sup_{\nu \in (0,1]}\P \left( \|\mf^{\nu} \|_{C^\gamma([0,T];H^{-s-2})} > C_{\delta,\gamma,T} \right) \le \tilde C_{p,T} R^{-2p} + C_{\delta,\gamma,T}^{-1} \hat C_{\gamma,p, T} R.
	\end{equation*}
	To conclude, we first choose a large enough $R>0$ such that $\tilde C_{p,T} R^{-2p} < \delta/2$ and then a large enough $C_{\delta,\gamma,T}$ so that $C_{\delta,\gamma,T}^{-1} \hat C_{\gamma,p, T} R < \delta/2$.	
\end{proof}

In order to prove the tightness for the family $(\mf^\nu)_{\nu>0}$, we rely on a compactness result with weighted Sobolev spaces, which we need here because we work on the whole space $\R^d$. We take the smooth weight $w(x)=(1+|x|^2)^{-d/2-1}$ and, for $s\in \R$, we introduce the weighted Sobolev space $H^s_w$ defined as the closure of $C^\infty_c(\R^d)$ with respect to the norm
\begin{equation*}
\|f\|_{H^s_w} \coloneqq \|fw \|_{H^s}.
\end{equation*}
Properties of these weighted spaces can be found for example in \cite[Lemma A.4]{BagnaraGaleatiMaurelli2024}. In particular, we have the following lemma:

\begin{lemma} \label{lem:compact_embedding}
	For every $r, m\in \R$ with $r\ge m$, $T>0$, $\gamma>0$ and $\eps>0$, the following embedding is compact
	\begin{equation*}
		L^\infty([0,T];H^r) \cap C^\gamma([0,T];H^m)  \hookrightarrow C([0,T];H^{r-\eps}_w)
	\end{equation*}
\end{lemma}

The proof is completely analogous to that of \cite[Lemma 3.5]{BagnaraGaleatiMaurelli2024}.
As a consequence, we gain the tightness of the viscous approximations. For convenience, we work with the series representation of the Wiener process (from \cref{lem:noise_series}): for given $\sigma_k$ orthonormal basis of $H^{\frac d2 + \alpha}_{\div}$, we have $W^\cK = \sum_k \sigma_k W^k$ with $W^k$ independent real $(\mathcal{F}_t)_t$-Brownian motions.

\begin{lemma}\label{lem:tightness} 
The family of laws of $\left((\mf^\nu, (W^k)_{k\ge 1})\right)_{\nu \in (0,1]}$ is tight in $C([0,T]; \allowbreak H^{-s-\eps}_w)\allowbreak \times C([0,T];\R)^\N$.
\end{lemma}

\begin{proof}
	Since $(W^k)_{k\ge 1}$ does not depend on $\nu$, it is enough to prove the tightness of $(\mf^\nu)_{\nu \in (0,1]}$ in $C([0,T]; H^{-s-\eps}_w)$.
	By \cref{lem:compact_embedding}, bounded sets in $L^\infty([0,T];H^{-s}) \cap C^\gamma([0,T];H^{-s-2})$ are compact in $C([0,T];H^{-s-\eps}_w)$. As a consequence, uniform in $\nu\in(0,1]$ boundedness in probability in $L^\infty([0,T];H^{-s}) \cap C^\gamma([0,T];H^{-s-2})$ implies the desired tightness. This uniform boundedness follows easily by \cref{lem:sup_by_stoch_Gronwall}, via Markov inequality, and \cref{cor:uniform_Holder_continuity}.
\end{proof}

\subsection{Passage to the limit}
Thanks to the tightness shown in the previous \cref{lem:tightness}, we are in the position to apply the classical Prokhorov and Skorokhod's machinery to show weak existence for the passive vector equation \eqref{eq:KK}, by passing to the limit as $\nu\to 0$ in a new probability space.

By Prokhorov's thereom, the tightness in \cref{lem:tightness} implies the existence of a sequence $\nu^n \to 0$ (as $n \to \infty$) such that the family of laws $\left((\mf^{\nu_n}, (W^k)_{k\ge 1})\right)_{n\in \N}$ converges weakly in the sense of probabilities. Then, by Skorokhod's theorem, there exists a new complete probability space $( \tilde \Omega, \tilde \cA, \tilde \P)$ and new random variables $\left(( \tilde \mf^{\nu_n}, (\tilde W^{n,k})_{k\ge 1})\right)_{n\in \N}$, $( \tilde \mf, (\tilde W^k)_{k\ge 1})$ defined therein such that 
\begin{align*}
    &\operatorname{Law}_{\P}\left((\mf^{\nu_n}, (W^k)_{k\ge 1})\right)= \operatorname{Law}_{\tilde \P} \left(( \tilde \mf^{\nu_n}, (\tilde W^{n,k})_{k\ge 1})\right) \quad \text{for every }n\in \N,\\
    &( \tilde \mf^{\nu_n}, (\tilde W^{n,k})_{k\ge 1})\to ( \tilde \mf, (\tilde W^k)_{k\ge 1}) \quad  \tilde \P\text{-a.s.}
\end{align*}
as $C([0,T];H^{-s-\eps}_w)\times C([0,T];\R)^\N$-valued random variables.

We denote by $(\tilde \cG_t)_{t\in[0,T]}$ the filtration generated by $\tilde \mf$, $(\tilde W^k)_{k \ge 1}$ and the $\tilde \P$-null sets, and we define $\tilde \cF_t \coloneqq \cap_{s>t} \tilde \cG_s$; clearly the filtration $(\tilde \cF_t)_{t\in [0,T]}$ satisfies the standard assumption. Analogously, we can construct the filtrations $(\tilde \cG^n_t)_{t\in[0,T]}$ and $(\tilde \cF^n_t)_{t\in[0,T]}$ obtained with the processes $\tilde \mf^{\nu_n}$ and $(\tilde W^{n,k})_{k \ge 1}$.

The following two lemmas are technical but quite classical; hence, we decide to present them omitting their proofs. 

\begin{lemma}\label{lem:Wiener_processes}
	For every $n\in\mathbb{N}$, $(\tilde W^{n,k})_{k\ge 1}$ is a sequence of mutually independent one-dimensional Brownian motions adapted to the filtration $(\tilde \cF^n_t)_{t\in[0,T]}$. Similarly, $(\tilde W^k)_{k\ge 1}$ is a sequence of mutually independent one-dimensional Brownian motions adapted to the filtration $(\tilde \cF_t)_{t\in[0,T]}$.
\end{lemma}

We call
\begin{align*}
    \tilde{W}^{\cK,n} = \sum_{k\ge 1} \sigma_k \tilde W^{n,k},\quad \tilde{W}^\cK = \sum_{k\ge 1} \sigma_k \tilde W^{k}.
\end{align*}
By \cref{lem:noise_series}, $\tilde{W}^{\cK,n}$, $n\in\N$, and $\tilde{W}^\cK$ are Wiener processes with covariance matrix $Q$ given in \eqref{eq:Q_spectrum}.

\begin{lemma}\label{lem:viscous_weak_solutions}
	For every $n\in \N$, $(\tilde \Omega,  \tilde \cA,(\tilde \cF^n_t)_t, \tilde \P,\tilde{W}^{\cK,n}, \tilde \mf^{\nu_n})$ is a weak solution to the viscous model \eqref{eq:SPDE_viscous}. 
\end{lemma}

In the following lemma, we show weak existence for the inviscid SPDE by passing to the limit as $n\to \infty$, closing the proof of existence in \cref{thm:main}:

\begin{theorem}
	$(\tilde \Omega,  \tilde \cF,(\tilde \cF_t)_t,\tilde \P, \tilde W^\cK, \tilde \mf)$ is a weak solution to \eqref{eq:KK}, satisfying $\tilde\mf\in L^\infty_t(H^{-s})$ $\P$-a.s. and the bound \eqref{eq:reg_gain}.
\end{theorem}
\begin{proof}
	Let $\varphi \in C^\infty_c(\R^d)$, by testing the weak solutions given in \cref{lem:viscous_weak_solutions} with $\varphi$ we obtain
	\begin{equation}\label{eq:viscous_prelimit}
		\langle \tilde \mf^{\nu_n}_t,\varphi\rangle - \langle \tilde  \mf^{\nu_n}_0,\varphi\rangle + \int_0^t \langle  B[\tilde \mf_r^{\nu_n}] \dd \tilde  W^{\cK,n}_r, \varphi\rangle
		= \left( \nu_n +\frac{c_0}{2} \right)\int_0^t \langle  \Delta \tilde \mf^{\nu_n}_r, \varphi \rangle \dd r.
	\end{equation}
	We want to pass every term in \eqref{eq:viscous_prelimit} to the limit $n\to\infty$. We focus only on the stochastic integral, being the other terms very simple to treat. To this end, we aim to apply \cite[Lemma 4.3]{bagnara2023no}, which will supply the convergence 
	\begin{equation} \label{eq:stoch_int_convergence}
		\sup_{t \in [0,T]} \left| \int_0^t \langle  B[\tilde \mf_r^{\nu_n}] \dd \tilde  W^{\cK,n}_r, \varphi\rangle - \int_0^t \langle  B[\tilde \mf_r] \dd \tilde  W^\cK_r, \varphi\rangle \right| \to 0
	\end{equation}
	in probability, provided we are able to show
	\begin{equation}\label{eq:stoch_int_conv_suff} 
		 \int_0^T \sum_{k\ge 1} \left| \langle  B[\tilde \mf_r^{\nu_n}-\tilde \mf_r]\sigma_k, \varphi\rangle  \right|^2 \dd r
		  \to 0
	\end{equation}
	in probability. Let $\chi \in C^\infty_c(\R^d,\R)$ be such that $\chi \equiv 1$ on the support of $\varphi$. To show \eqref{eq:stoch_int_conv_suff}, we apply \cref{lem:HS_homogeneous,lem:products} and get
	\begin{align*}
		\int_0^T &\sum_{k\ge 1} \left| \langle  B[\tilde \mf_r^{\nu_n}-\tilde \mf_r]\sigma_k, \varphi\rangle  \right|^2 \dd r \\
		& = \int_0^T \sum_{k\ge 1} \left| \langle  B[\chi(\tilde \mf_r^{\nu_n}-\tilde \mf_r)]\sigma_k, \varphi\rangle  \right|^2 \dd r \\
		&\le \|\varphi\|^2_{H^{s+\eps+1 }} \int_0^T \| B[\chi(\tilde \mf_r^{\nu_n}-\tilde \mf_r)]\|^2_{\HS(H^{d/2+\alpha}_{\div};H^{-s-\eps-1})} \dd r \\
		& \lesssim \|\varphi\|^2_{H^{s+\eps+1}} \int_0^T \| \chi(\tilde \mf_r^{\nu_n}-\tilde \mf_r)\|^2_{H^{-s-\eps}} \dd r,\\
        & \lesssim \|\varphi\|^2_{H^{s+\eps+1}} \left\|\frac\chi w\right\|^2_{H^{d/2+\alpha}} \int_0^T \| w(\tilde \mf_r^{\nu_n}-\tilde \mf_r)\|^2_{H^{-s-\eps}} \dd r,
	\end{align*}
	with the integral converging to $0$ (as $n\to\infty$) $\tilde \P$-a.s.\ (since $\tilde \mf^{\nu_n}\to  \tilde \mf $  $\tilde \P$-a.s.\ in $C([0,T];H^{-s-\eps}_w)$). This proves \eqref{eq:stoch_int_conv_suff} and so \eqref{eq:stoch_int_convergence}. Up to a subsequence extraction, we can pass to the $\tilde \P$-a.s.\ limit in \eqref{eq:viscous_prelimit} obtaining $\tilde \P$-a.s.
	\begin{equation}\label{eq:tested}
		\langle \tilde \mf_t,\varphi\rangle - \langle \tilde  \mf_0,\varphi\rangle + \int_0^t \langle  B[\tilde \mf_s] \tilde  W^\cK_s, \varphi\rangle 
		=  \frac{c_0}{2} \int_0^t \langle  \Delta \tilde \mf_s, \varphi \rangle \dd r,\quad \forall t\in [0,T],
	\end{equation}
    for every $\varphi \in C^\infty_c(\R^d)$. By a density argument on $\varphi$, we can show \eqref{eq:tested} on a $\tilde\P$-exceptional set independent of $\varphi$, thus proving \eqref{eq:tested}. Finally, by a standard semi-continuity argument, from \eqref{eq:sup_time_nu} we get for $p\in (0,1)$
    \begin{align*}
        \E \left(\sup_{t\in[0,T]} \|\tilde \mf_t\|^{2p}_{\dot H^{-s}}\right)
 		\le c_p \|\mf_0\|^{2p}_{\dot{H}^{-s}}  e^{p\rho_{d,s,\alpha}T},
    \end{align*}
    implying $\tilde \mf\in L^\infty_t(\dot{H}^{-s})$ $\tilde \P$-a.s., and from \eqref{eq:reg_gain_nu} we get \eqref{eq:reg_gain}. The proof is complete.
\end{proof}

\section{Proof of Uniqueness}\label{sec:ProofUniqueness}

In this Section we prove the uniqueness statement in \cref{thm:main}.

We consider the family of smooth kernels $(G_\delta)_{\delta>0}$ given in Fourier modes by
\begin{align*}
    \widehat{G_\delta}(n) = |n|^{-2s}\varphi_\delta(n),
\end{align*}
where $(\varphi_\delta)_{\delta>0}$ is a non-decreasing family of smooth functions with $0\le \varphi_\delta\le 1$ on $\R^d$, $\varphi_\delta(n)=1$ for $\delta\le |n|\le 1/\delta$, $\varphi_\delta(n)=0$ for $|n|\le \delta/2$ and $|n|\ge 2/\delta$, and, for $h=1,2$,
\begin{align}\label{eq:varphi_delta_bd}
|D^{(h)}\varphi_\delta(n)|\lesssim |n|^{-h},\quad \forall n\in\R^d,
\end{align}
uniformly in $\delta>0$ small (for example, take $\rho:[0,\infty)\to [0,1]$ smooth with $\rho=1$ on $[0,1]$, $\rho=0$ on $[2,\infty)$ and take $\varphi_\delta(n)=\rho(\delta(|n|\vee |n|^{-1}))$).

We use a localization argument to deal with integrability in $\omega$. Given $M$ solution to \eqref{eq:KK}, for $R>0$, we call (being $M\in C_t(H^{-s-2})$ $\P$-a.s.)
\begin{align*}
    &\tau_R = \inf\{t\ge 0\mid \|\mf_t\|_{H^{-s-2}}\ge R\} \wedge T,\\
    &\mf^R_t = \mf_{t\wedge \tau_R}.
\end{align*}

\begin{lemma}\label{lem:limit_uniq}
    Assume $d\ge 2$, $s\in (0,d/2)$ and $\alpha\in (0,1)$, let $\mf_0$ be in $\dot{H}^{-s}$. Let $\mf$ be a solution to \eqref{eq:KK} (in the sense of \cref{def:sol}) which is also in $L^2_{t,\omega}(H^{-s}))$. For every $\delta>0$ and $R>0$, we have
    \begin{align}
    \begin{aligned}\label{eq:eq norm approx}
        \E \int_{\R^d} &|\widehat{\mf}^R_t(n)|^2 \widehat{G_\delta}(n) \dd n -\int_{\R^d} |\widehat{\mf}_0(n)|^2 \widehat{G_\delta}(n) \dd n \\
        &= (2\pi)^{-d/2}\E \int_0^{t\wedge \tau_R} \int_{\R^d} \overline{\widehat{\mf}}_r(n) \cdot \mathbb H_\delta(n)\widehat{\mf}_r(n) \dd n \dd r,
    \end{aligned}
    \end{align}
    where
    \begin{align*}
        \mathbb H_\delta(n) &= \mathbb F_{\gamma_\delta}(n) -(2\pi)^{d/2}c_0|n|^{-2s+2}I \\
        &= \int_{\R^d} \langle n-k \rangle^{-d-2\alpha} (\varphi_\delta(k)|k|^{-2s}-\varphi_\delta(n)|n|^{-2s}) |P^\perp_{n-k}k|^2 I \dd k \\
        &\quad +(d-1)\int_{\R^d} \langle n-k \rangle^{-d-2\alpha} \varphi_\delta(k)|k|^{-2s} kk^\top \dd k \\
        &\quad -2\int_{\R^d} \langle n-k \rangle^{-d-2\alpha} \varphi_\delta(k)|k|^{-2s} P^\perp_{n-k} kk^\top \dd k
    \end{align*}
    and $\mathbb F_{\gamma_\delta}$ is defined in \eqref{eq:F_gamma} with $\gamma_\delta(\dd n) =\hat{G}_\delta (n)\dd n$.
\end{lemma}

\begin{proof}
    By It\^o formula applied to the functional $\mf\mapsto \langle \mf, G_\delta \ast \mf\rangle = \| \widehat\mf \|^2_{L^2_{\gamma_\delta}}$, we get
    \begin{align}
    \begin{aligned}\label{eq:approx_norm}
        \dd \langle \mf_t, G_\delta \ast \mf_t \rangle &= 2\sum_k \langle B[\mf_t]\sigma_k,G_\delta \ast\mf_t\rangle \dd W^k_t \\
        &\quad +\left(\langle \mf_t, c_0\Delta G_\delta \ast \mf_t \rangle +  \|\cF B[\mf_t]  \|^2_{\HS(H^{d/2+\alpha}_{\div}; L^2_{\gamma_\delta})} \right) \dd t.
    \end{aligned}
    \end{align}
    By \cref{lem:HS_homogeneous}, we have
    \begin{align*}
    	\sum_k \langle B[\mf_t]\sigma_k,G_\delta \ast\mf_t\rangle^2 
    	&\le \|G_\delta \ast\mf_t \|^2_{H^{s+1}} \sum_{k} \| B[\mf_t]\sigma_k \|^2_{H^{-s-1}}\\
    	&\lesssim_\delta \|\mf_t\|^2_{H^{-s-2}} \|\mf_t\|^2_{H^{-s}}.
    \end{align*}
    Since $M$ is in $L^2_{t,\omega}(H^{-s})$, taking the expectation in \eqref{eq:approx_norm} up to the stopping time $\tau_R$, the local martingale term disappears.
    By \cref{lem:FB_HS_norm_in_L2}, we have
    \begin{align}
    \begin{aligned}\label{eq:eq norm approx stopped}
        &\E \int_{\R^d} |\widehat{\mf}^R_t(n)|^2  \widehat{G_\delta}(n) \dd n - \int_{\R^d} |\widehat{\mf}^R_0(n)|^2  \widehat{G_\delta}(n) \dd n \\
        &= \E \int_0^{t\wedge\tau_R} \int_{\R^d} \overline{\widehat{\mf}}_r(n) \cdot \left(-c_0|n|^2\widehat{G_\delta}(n)I +(2\pi)^{-d/2}\mathbb F_{\gamma_\delta}(n) \right) \widehat{\mf}_r(n) \dd n \dd r,
    \end{aligned}
    \end{align}
    that is \eqref{eq:eq norm approx}.
\end{proof} 

\begin{lemma}
    Fix $R>0$. For $d\ge 2$, $s\in (1,d/2)$ and $\alpha\in (0,1)$, for every $\mf$ in $L^2_{t,\omega}(\dot{H}^{-s+1-\alpha}_{\div})$, we have
    \begin{align*}
        \lim_{\delta\to 0} \E \int_0^{t\wedge \tau_R} \int_{\R^d} \overline{\widehat{\mf}}_r(n) &\cdot \mathbb H_\delta(n)\widehat{\mf}_r(n) \dd n \dd r \\
        &= \E \int_0^{t\wedge \tau_R}\int_{\R^d} \overline{\widehat{\mf}}_r(n) \cdot \mathbb H(n)\widehat{\mf}_r(n) \dd n \dd r,
    \end{align*}
    where $\mathbb H = \mathbb H_{d,s,\alpha}$ is defined by \eqref{eq:H_def}.
\end{lemma}

\begin{proof}
    The proof is similar to the proof of \cite[Proposition 3.2]{GGM2024}, taking into account that $\mf$ is divergence-free.

    Since $n\cdot \widehat{\mf}(n)=0$, we can write
    \begin{align*}
        \overline{\widehat{\mf}}(n) \cdot \mathbb H_\delta(n) \widehat{\mf}(n) = \overline{\widehat{\mf}}(n) \cdot \mathbb K_\delta(n) \widehat{\mf}(n),
    \end{align*}
    where
    \begin{align*}
        &\mathbb K_\delta(n) = \int_{\R^d} \langle n-k \rangle^{-d-2\alpha} (\varphi_\delta(k)A_{n-k}(k)-\varphi_\delta(n)A_{n-k}(n)) \dd k ,\\
        &A_{n-k}(\eta) = |\eta|^{-2s} (|P^\perp_{n-k}\eta|^2I +(d-1)\eta\eta^\top -2P^\perp_{n-k}\eta \eta^\top).
    \end{align*}
    An analogous identity holds for $\mathbb H$ removing $\varphi_\delta$ from the integral (in fact, this is the expression in Fourier modes of the identity \eqref{eq:time_evolution_sss}). By dominated convergence theorem, for every $n\neq 0$, $\mathbb K_\delta(n)$ tends to $\mathbb K(n)$ as $\delta\to 0$. Since $\mf$ is in $L^2_{t,\omega}(\dot{H}^{-s+1-\alpha})$, it is enough to show that $|\mathbb K_\delta(n)| \lesssim |n|^{-2s+2-2\alpha}$ uniformly in $\delta$.

    For $|n|\le 1$, $|\mathbb F_{\gamma_\delta}(n)|\lesssim 1$ (see the proof of \cref{lem:HS_homogeneous}), therefore $|\mathbb{K}_\delta(n)|\lesssim |n|^{-2s+2} \le |n|^{-2s+2-2\alpha}$, as wanted for $|n|\le 1$.

    For $|n|>1$, we split the integral in $\mathbb K_\delta$ into three domains:
    \begin{align*}
        D_1&=\{|n-k|> |n|/2,\,|k|> |n|/2\},\\
        D_2&=\{|n-k|> |n|/2,\,|k|\le |n|/2\},\\
        D_3&=\{|n-k|\le |n|/2\},
    \end{align*}
    and call $\mathbb K_\delta^{(j)}$ the corresponding integral on $D_j$, $j=1,2,3$.
    Since $s>1$ we have $|A_{n-k}(k)|+|A_{n-k}(n)|\lesssim (|k|\wedge |n|)^{-2s+2}$. Therefore, on $D_1$ we get
    \begin{align*}
        |\mathbb K_\delta^{(1)}(n)|
        &\lesssim \int_{|n-k|>|n|/2} \langle n-k \rangle^{-d-2\alpha} |n|^{-2s+2} \dd s \\
        &\lesssim |n|^{-2s+2} |n|^{-2\alpha} = |n|^{-2s+2-2\alpha}.
    \end{align*}
    On $D_2$ we get
    \begin{align*}
        |\mathbb K_\delta^{(2)}(n)|
        &\lesssim \langle n \rangle^{-d-2\alpha} \int_{|k|\le |n|/2} |k|^{-2s+2} \dd s \\
        &\lesssim |n|^{-d-2\alpha} |n|^{d-2s+2} = |n|^{-2s+2-2\alpha}.
    \end{align*}
    On $D_3$, we expand the term $\varphi_\delta A_{n-k}$ around $n$.
    The gradient term $(k-n)\cdot \nabla[\varphi_\delta A_{n-k}](n)$ is an odd function in $n-k$ (and $D_3$ is symmetric in $n-k$), therefore we have
    \begin{align*}
        \int_{D_3} \langle n-k \rangle^{-d-2\alpha} (k-n) \cdot \nabla[\varphi_\delta A_{n-k}](n) \dd k =0.
    \end{align*}
    and so $\mathbb K_\delta^{(3)}(n)$ reads
    \begin{align*}
        \mathbb K_\delta^{(3)}(n) = \int_{D_3} \langle n-k \rangle^{-d-2\alpha} (k-n)\cdot \int_0^1 D^2[\varphi_\delta A_{n-k}](\theta k+(1-\theta)n) \dd\theta \,(k-n) \dd k.
    \end{align*}
    The derivatives of $A_{n-k}$ satisfy, for $h=0,1,2$,
    \begin{align*}
        |D^{(h)} A_{n-k}(\eta)| \lesssim |\eta|^{-2s+2-h},
    \end{align*}
    and so (recall \eqref{eq:varphi_delta_bd}) $|D^2[\varphi_\delta A_{n-k}](\eta)| \lesssim |\eta|^{-2s}$. Hence we get
    \begin{align*}
        |\mathbb K_\delta^{(3)}|
        &\lesssim \int_{|n-k|\le |n|/2} \langle n-k \rangle^{-d-2\alpha} |n-k|^2 \sup_{\theta\in [0,1]} |\theta k+(1-\theta)n|^{-2s} \dd k \\
        &\lesssim |n|^{-2s} \int_{|n-k|\le |n|/2} |n-k|^{-d-2\alpha+2} \dd k \\
        &\lesssim |n|^{-2s} |n|^{-2\alpha+2} = |n|^{-2s+2-2\alpha}.
    \end{align*}
    Putting together the bounds on $\mathbb K_\delta^{(j)}$, $j=1,2,3$, we get that $|\mathbb K_\delta(n)|\lesssim |n|^{-2s+2-2\alpha}$ for $|n|>1$, as wanted. The proof is complete.
\end{proof}

\begin{corollary}
    Assume that $d\ge 3$, $s\in (1,d/2)$ and $\alpha\in (0,1)$ and $\alpha\in (0,1/2)$, let $\mf_0$ be in $\dot{H}^{-s}$, fix $R>0$. Let $\mf$ be a solution to \eqref{eq:KK} (in the sense of \cref{def:sol}) which is also in $L^2_{t,\omega}(\dot{H}^{-s+1-\alpha})$. Then we have
    \begin{align}\label{eq:eq norm}
        \E\|\mf^R_t\|_{\dot{H}^{-s}}^2 -\|\mf_0\|_{\dot{H}^{-s}}^2 = (2\pi)^{-d/2} \E \int_0^{t\wedge \tau_R} \int_{\R^d} \overline{\widehat{\mf}}_t(n) \cdot \mathbb H(n)\widehat{\mf}_t(n) \dd n,
    \end{align}
    where $\mathbb H$ is defined as in \eqref{eq:H_def}; in particular $\mf^R$ is in $L^\infty_t(L^2_\omega(\dot{H}^{-s}))$.
\end{corollary}

\begin{proof}
    This formula is obtained by passing to the limit $\delta\to 0$ in formula \eqref{eq:eq norm approx} and using \cref{lem:limit_uniq}. The formula implies that $\E\|\mf^R_t\|_{\dot{H}^{-s}}^2$ is bounded uniformly in time.
\end{proof}

We are ready to conclude the proof of \cref{thm:main}. For fixed $R>0$, from formula \eqref{eq:eq norm} and the main bound \eqref{eq:main_bd_integral} on $\mathbb H$, we get
\begin{align*}
    \E\|\mf^R_t\|_{\dot{H}^{-s}}^2 -\|\mf_0\|_{\dot{H}^{-s}}^2 &\le \rho_{d,s,\alpha} \E \int_0^{t\wedge \tau_R} \|\mf_t\|_{\dot{H}^{-s}}^2 \dd r\\
    &\le \rho_{d,s,\alpha} \int_0^t \E \|\mf^R_t\|_{\dot{H}^{-s}}^2 \dd r.
\end{align*}
If $\mf_0=0$, then by Gr\"onwall inequality, for every $t$, $\mf^R_t=0$ $\P$-a.s., hence $\mf^R=0$ $\P$-a.s. (by time continuity of paths). Therefore, taking $R\to\infty$, we have $\mf=0$ $\P$-a.s.. By linearity of \eqref{eq:KK}, this shows uniqueness. The proof of \cref{thm:main} is complete.


\appendix
\section{Auxiliary and Technical Aspects}\label{sec:Appendix} 

\subsection{It\^o-Stratonovich correction}\label{sec:Ito-Stratonovich_correction} 
Here we show formally the expression \eqref{eq:noise_term_ito+lap}. We recall the series representation in \cref{lem:noise_series} and define $B_k[w]=B[w]\sigma_k$. The link between It\^o and Stratonovich integrals in equation \eqref{eq:KK_Strat} is formally
\begin{equation*}
    -\int_0^t B[\mf_s] \circ \dd W^\cK_s = -\int_0^t B[\mf_s]  \dd W^\cK_s +\frac12 \int_0^t \sum_{k\ge 1} B^2_k[\mf_s] \dd s,
\end{equation*}
hence \eqref{eq:noise_term_ito+lap} is equivalent to
\begin{equation}\label{eq:ItoStrat_formal}
    \sum_{k\ge 1} B^2_k \mf = c_0 \Delta \mf,
\end{equation}
which we will now show, assuming smoothness of $\mf$ and $Q$. We have
\begin{equation*}
    B^2_k \mf = \sigma_k \cdot \nabla \left( \sigma_k \cdot \nabla\mf - \mf \cdot \nabla\sigma_k \right) - \left( \sigma_k \cdot \nabla\mf - \mf \cdot \nabla\sigma_k \right) \cdot \nabla \sigma_k.
\end{equation*}
Dropping for a moment the $k$ subscript and exploiting the divergence free condition of both $\mf$ and $\sigma$, we obtain
\begin{align*}
    &\left[\sigma \cdot \nabla \left( \sigma \cdot \nabla\mf - \mf \cdot \nabla\sigma \right) - \left( \sigma \cdot \nabla\mf - \mf \cdot \nabla\sigma_k \right) \cdot \nabla \sigma \right]^h\\
    &\qquad=\sigma^i\partial_i \left( \sigma^j \partial_j \mf^h - \mf^j\partial_j \sigma^h \right) - \left( \sigma^i \partial_i \mf^j - \mf^i\partial_i \sigma^j \right) \partial_j \sigma^h\\
    &\qquad=\partial_i \left( \sigma^i \sigma^j \partial_j \mf^h - 2\sigma^i\mf^j\partial_j \sigma^h  + \mf^i \partial_j( \sigma^j  \sigma^h) \right).
\end{align*}
Then, summing over $k$, we have
\begin{equation*}
    \sum_{k\ge 1} B^2_k \mf = \div \left( Q(0)\nabla \mf \right) +2 \div(\mf\cdot \nabla Q(0)) + \div(\mf \div (Q(0))),
\end{equation*}
where, we used that
\begin{align*}
    \sum_{k\ge 1} \sigma^i_k(x)\partial_j\sigma^h_k (x) 
    &= \sum_{k\ge 1} \partial_{y_j}\left(\sigma^i_k(x)\sigma^h_k (y)\right)\big|_{y=x}\\
    &= \partial_{y_j}Q(x-y)^{ih}\big|_{y=x}
    = \partial_{y_j}Q(x-y)^{ih}\big|_{y=x} = - \partial_j Q(0)^{ih}.
\end{align*}
Being $Q$ an even function, we formally have $\nabla Q=0$. Note, however, that in our setting $Q$ is not rigorously differentiable at the origin. Finally, recalling that $Q(0) =c_0\operatorname{I} $, formula \eqref{eq:ItoStrat_formal} is proven.

\subsection{Technical lemmas and proofs}

Here we give some technical lemmas and proofs.

\begin{proof}[Proof of \cref{lem:rep_kernel}]
	We can decompose the operator $\cQ^\frac12$ as $\cQ^\frac12= \cT\cL$, where $\cL$ is the Leray projector, i.e.\ the Fourier multiplier operator associated to the matrix $\left(I-\frac{n n^\top}{|n|^2}\right)$, and $\cT$ is the operator associated to the remaining multiplier $\bracn^{-\frac d2-\alpha}$. The operator $\cL$ maps $L^2$ into $L^2_{\div}$ surjectively, while $\cT$ is an isometric bijection between $H^s_{\div}$ and $H^{s+d/2+\alpha}_{\div}$ for every $s\in\R$ (a detailed proof would be similar to the corresponding part of the proof of \cref{lem:topological_dual}). This proves the first part of the lemma.
	
	Concerning the second part, by the definition of pseudo-inverse operator, we have
	\begin{align*}
		\cQ^{-1/2} u 
		&\coloneqq \argmin \left\{\|v\|_{L^2}:\; v\in L^2,\, \cQ^\frac12 v = u\right\}\\
		&= \argmin \left\{\|v\|_{L^2}:\; v\in L^2,\,  \cL v = \cT^{-1} u\right\},
	\end{align*}
	and, by Helmholtz–Leray decomposition, the minimizer is the only element $v$ which is already divergence free (i.e.\ $\cL v = v$). It follows $\cQ^{-1/2} u = \cT^{-1} u$ and, in particular, by the isometry property of $\cT$, we have
	\begin{equation*}
		\| \cQ^{-\frac 12 }u\|_{L^2}
		=\| \cT^{-1}u\|_{L^2}=\|u\|_{H^{\frac d2 + \alpha}},
	\end{equation*}
	which completes the proof.
\end{proof}

\begin{lemma}\label{lem:Q_not_trace_class}
	The operator $\cQ: L^2 \to L^2$ is not trace class.
\end{lemma}
\begin{proof}
	Assume by contradiction that $\cQ$ is trace class. Then, $\cQ$ is compact (\cite[Proposition 18.6 and Corollary 18.7]{Conway2000}), hence, also its (not relabeled) restriction $\cQ: L^2_{\div} \to L^2_{\div}$ is compact. By the spectral theorem, there exists an orthonormal basis of $L^2_{\div}$ made of eigenvectors of $\cQ$. Let $v\in L^2_{\div}$ be an eigenvector. By Fourier transforming the equality $\cQ v = \lambda v$, we obtain
	\begin{equation*}
		(2\pi)^{\frac d2}\bracn^{-d-2\alpha} \widehat{v}(n)
		= \lambda \widehat{v}(n),
	\end{equation*}
	which requires $\widehat{v}$ to be supported on a single level set of $\bracn^{-d-2\alpha}$. However, every level set of $\bracn^{-d-2\alpha}$ has zero Lebesgue measure, forcing $v=0$. This is a contradiction.
\end{proof}

\begin{proof}[Proof of \cref{lem:topological_dual}]
    For both $\dot H^{s+1}_\vee(\R^d)$ and $\dot H^{s-1}_\wedge (\R^d)$, the singularity at the origin in the Fourier weight is $|n|^{2s}$. In particular, we can replicate the proof of \cite[Proposition 1.34]{Bahouri2011} to show that $\dot H^{s+1}_\vee(\R^d)$ and $\dot H^{s-1}_\wedge (\R^d)$ are Hilbert spaces.
    
	Concerning duality, let us introduce the operator $\cT$ associated with the Fourier multiplier $(1\vee |n|^2)$ and its inverse $\cT^{-1}$ associated to the inverse of such multiplier. For $v\in\dot H^{s+1}_\vee$ and $w\in \dot H^{s-1}_\wedge$, we have
	\begin{align*}
		\|\cT v\|^2_{\dot H^{s-1}_\wedge}
		&= \int_{\R^d} |n|^{2s}\left(1\wedge|n|^{-2}\right)\left(1\vee|n|^{2}\right)^2 |\widehat{v}(n)|^2 \dd n\\
		&=\int_{\R^d} |n|^{2s}\left(1\vee|n|^{2}\right) |\widehat{v}(n)|^2 \dd n =\|v\|^2_{\dot H^{s+1}_\vee} ,\\
		\|\cT^{-1}w\|^2_{\dot H^{s+1}_\vee}
		&= \int_{\R^d} |n|^{2s}\left(1\vee|n|^{2}\right)\left(1\vee|n|^{2}\right)^{-2} |\widehat{w}(n)|^2 \dd n\\
		&= \int_{\R^d} |n|^{2s}\left(1\wedge|n|^{-2}\right) |\widehat{w}(n)|^2 \dd n = \|w\|^2_{\dot H^{s-1}_\wedge},
	\end{align*}
	showing that $\cT:\dot H^{s+1}_\vee \to \dot H^{s-1}_\wedge$ is a bijection and an isometry.
	Moreover, being for $v\in\dot H^{s+1}_\vee$ and $w\in \dot H^{s-1}_\wedge$
	\begin{align*}
		\int_{\R^d} |n|^{2s} \widehat v(n)\cdot\overline{\widehat w} (n)\dd n
		&=\int_{\R^d} |n|^{2s}\left(1\wedge|n|^{-1}\right)\left(1\vee|n|\right) \widehat v(n)\cdot\overline{\widehat w} (n) \dd n \\
		&\le \|v\|_{\dot H^{s+1}_\vee} \|w\|_{\dot H^{s-1}_\wedge},
	\end{align*}  
	we can extend the scalar product in $\dot H^s$ to a duality between $\dot H^{s+1}_\vee$ and $\dot H^{s-1}_\wedge$.
	
	Finally, we characterize the topological dual $(\dot H^{s+1}_\vee)^*$. Let $l\in (\dot H^{s+1}_\vee)^*$, then by Riesz representation theorem, there exists $\cR l \in \dot H^{s+1}_\vee$ such that for every $v \in \dot H^{s+1}_\vee$
	\begin{align*}
		_{(\dot H^{s+1}_\vee)^*} \langle l,v \rangle_{\dot H^{s+1}_\vee}
		&= _{\dot H^{s+1}_\vee} \langle \cR l,v \rangle_{\dot H^{s+1}_\vee}\\
		&= \int_{\R^d} |n|^{2s} \vee |n|^{2s+2}\widehat{\cR l}(n)\cdot\overline{\widehat w} (n)\dd n\\
		&= \int_{\R^d} |n|^{2s} \widehat{\cT\cR l}(n)\cdot\overline{\widehat w} (n)\dd n 
		= _{\dot H^{s-1}_\wedge} \langle \cT \cR l,v \rangle_{\dot H^{s+1}_\vee},
	\end{align*}
	with $\|l\|_{(\dot H^{s+1}_\vee)^*} = \|\cR l\|_{\dot H^{s+1}_\vee}=\|\cT \cR l\|_{\dot H^{s-1}_\wedge}$, showing that $\cT \cR:(\dot H^{s+1}_\vee)^* \to \dot H^{s-1}_\wedge$ is a bijection and an isometry.
\end{proof}

\begin{acknowledgements}
    MB and MM acknowledge support and hospitality from the Bernoulli Center, EPFL, through the program `New developments and challenges in Stochastic Partial Differential Equations'. FG and MM acknowledge support and hospitality from Centro Internazionale per la Ricerca Matematica (Trento), Fondazione Bruno Kessler and Università degli Studi di Trento, through the Research in Pairs program. MM acknowledges support from the Italian Ministry of Research through the project PRIN 2022 `Noise in fluid dynamics and related models', project number I53D23002270006, and from Istituto Nazionale di Alta Matematica (INdAM) through the project GNAMPA 2024 `Fluidodinamica stocastica'. MB, FG and MM are members of Istituto Nazionale di Alta Matematica (INdAM), group GNAMPA.

    The authors thank Theodore Drivas for suggesting the problem.
\end{acknowledgements}

\bibliography{biblio.bib}{}
\bibliographystyle{plain}

\end{document}